\documentclass[final,12pt]{amsart}
\usepackage{amssymb, amsmath, amsthm,geometry}
\usepackage{mathrsfs}  

\usepackage{enumitem}

\usepackage{graphicx}
\usepackage[colorlinks=true, citecolor=blue, linkcolor=blue, urlcolor=blue]{hyperref}
\newgeometry{asymmetric, centering}
 \usepackage{xcolor}

\usepackage{ifdraft}
\ifoptionfinal{
\usepackage[disable]{todonotes}
}{
\usepackage[norefs, nocites]{refcheck}
\usepackage{graphicx}
\usepackage{epstopdf}

\usepackage[notref, notcite]{showkeys}
\usepackage[bordercolor=white, color=white]{todonotes}
}

\makeatletter\providecommand\@dotsep{5}\def\listtodoname{List of Todos}\def\listoftodos{\hypersetup{linkcolor=black}\@starttoc{tdo}\listtodoname\hypersetup{linkcolor=blue}}\makeatother
\allowdisplaybreaks[4]

\numberwithin{equation}{section}

\newtheorem{lemma}{Lemma}[section]
\newtheorem{proposition}{Proposition}[section]
\newtheorem{theorem}{Theorem}[section]

\newtheorem{definition}{Definition}[section]

\newtheorem{remark}{Remark}[section]

\theoremstyle{remark}

\newcommand{\bel}{\begin{equation} \label}
\newcommand{\ee}{\end{equation}}
\def\beq{\begin{equation}}
\def\eeq{\end{equation}}
\newcommand{\bea}{\begin{eqnarray}}
\newcommand{\eea}{\end{eqnarray}}
\newcommand{\beas}{\begin{eqnarray*}}
\newcommand{\eeas}{\end{eqnarray*}}

\def\R{\mathbb R}

\def\Z{\mathbb Z}
\def\N{\mathbb N}

\def\U{\mathcal U}

\def\Im{\mathrm{Im}}
\def\Char{\mathrm{Char}}
\def\WF{\mathrm{WF}}

\def\out{\mathrm{out}}

\def\sgn{\mathrm{sgn}}
\def\b{\backslash}
\def \into{\mathrm{in}}

\def \id{\mathrm{id}}

\def \comp{\mathrm{comp}}
\def \rem{\mathrm{rem}}
\def \fre{\mathrm{fre}}

\def\W{\mathcal W}

\def\D{\mathbb D}

\renewcommand{\leq}{\leqslant}

\def \e{\varepsilon}

\def \c{\boldsymbol}

\def\p{\partial}

\def\l{\left}
\def\r{\right}

\newcommand\rotom{\mho}

\DeclareMathOperator{\supp}{supp}

\date{Compiled \today}
\title%[A Semi-linear equation in Lorentzian Geometry]
[Inverse  problems of wave equations]{Unique inversion of smooth nonlinearity in semilinear wave equations on Lorentzian manifolds}

\author[X. Chen]{Xi Chen}
\address{Shanghai Center for Mathematical Sciences, Fudan University, Shanghai 200438, China;	School of Mathematical Sciences, Fudan University, Shanghai 200433, China;
	Center for Applied Mathematics, Fudan University, Shanghai 200433, China. }
\email{xi\_chen@fudan.edu.cn}
\author[S. Lu]{Shuai Lu}
\address{School of Mathematical Sciences, Fudan University, Shanghai 200433, China.}
\email{slu@fudan.edu.cn}
\author[R. Zhang]{Ruochong Zhang}
\address{School of Mathematical Sciences, Fudan University, Shanghai 200433, China.}
\email{22110180055@m.fudan.edu.cn}

\begin{document}
%\begin{abstract}
%\end{abstract}

\begin{abstract}
    We study the recovery of a smooth nonlinearity in semilinear wave equations on globally hyperbolic Lorentzian manifolds. Our approach combines higher-order linearization with distorted plane waves and Gaussian beams, using three-wave, four-wave, and m-wave interactions to recover different Taylor coefficients. 
\end{abstract}

\maketitle

\tableofcontents

\section{Introduction}

Consider  semilinear wave equations on a globally hyperbolic simple Lorentzian $4$-manifold $(M, g)$ with signature $(3, 1)$
\begin{align}\label{eqn : semilinear wave}  \Box_g u(x) + N(x, u(x)) = f(x)\end{align} with D'Alembertian $\Box_g$, nonlinearity $N $, source $f$, zero initial values and $x \in M$. Suppose that one can manipulate $f$ within a domain $ \mathcal{O}$ and observe the resulting waves in $\mathcal{O}$. The generic inverse problem of concern is whether the source-to-solution type local measurements uniquely determine the metric and the nonlinearity respectively. That is, \begin{align}\label{Q : g}
	\left\{\left(f^{(1)}, u^{(1)}|_{\mathcal{O}}\right)\right\} = \left\{\left(f^{(2)}, u^{(2)}|_{\mathcal{O}}\right)\right\}, N^{(1)} = N^{(2)} &\Longrightarrow g^{(1)} = g^{(2)} ?   \\
	\label{Q : N}
    \left\{\left(f^{(1)}, u^{(1)}|_{\mathcal{O}}\right)\right\} = \left\{\left(f^{(2)}, u^{(2)}|_{\mathcal{O}}\right)\right\}, g^{(1)} = g^{(2)} &\Longrightarrow   N^{(1)} = N^{(2)} ?
\end{align}

The time-dependence nature of  \eqref{eqn : semilinear wave} normally disables the classical unique continuation approach for inverse coefficient problems of hyperbolic equations with time-independent coefficients. 

In the landmark paper \cite{KLU}, Kurylev--Lassas--Uhlmann introduced the framework of higher order linearization, combined with the calculus of distorted plane waves, to recover $g$ from quadratic wave equations \begin{align}\label{eqn : quadratic wave equation}\Box_g u(x) +   u(x)^2 = f(x)\end{align} on $(M, g)$. The core of this methodology is to approach the nonlinear scattering waves by products of linear waves and extract the coefficients from the visible scattering data of such multilinear approximations. From this viewpoint, the metric $g$ addresses the propagation of linear waves, while the nonlinearity $N$ encodes the nonlinear scattering of linear waves.

Since the metric $g(x, dx)$ appears in a concrete format (i.e. a quadratic form in $dx$), question \eqref{Q : g} is, to some extent, well-understood via higher order linearization. See e.g. \cite{KLU, KLOU, LUW, Uhlmann-Wang-CPAM, FLO} for recovery of $g$ in various models. 

In contrast, the nonlinearity $N(x, u)$, in general, does not take on a specific form in $u$. If the nonlinearity is assumed to be a special power series  
\begin{align} \label{eqn : higher degree nonlinearity without potential}  h_2(x) u^2 + h_3(x) u^3 + h_4(x) u^4 + \cdots\end{align} with $h_k(x) \in C^\infty(M;\R)$ and $ \sum_{k=2}^\infty |h_k(x)| \neq 0$, \cite{LUW} showed that $N$ can be uniquely recovered from source-to-solution type local information. It was extended by \cite{Hintz-Uhlmann-Zhai-IMRN, Uhlmann-Zhang} to boundary value problems with Dirichlet-to-Neumann type local data. 

In this paper, we shall give a positive answer to question \eqref{Q : N} with source-to-solution type local data for more general nonlinearity $N\in C^\infty(M \times \mathbb{C})$ such that \begin{enumerate}
	\item[(i)]\label{i} $N(x, u)$ is smooth in $x$;  
	\item[(ii)]\label{ii} for any $x$, $N(x, \cdot)$ is not a linear function; 
	\item[(iii)]\label{iii} $N(x, u)$ is homogeneous in $u$;  
	\item[(iv)]\label{iv} $N(x, u)$ is analytic in $u$, and $N(x, u) \in \R$ if $u\in \R$.
\end{enumerate}

\subsection{Lorentzian geometry}

To present the geometric setup, we follow the standard notations in Lorentzian geometry from \cite{Beem-Ehrlich-Easley, O}.
\begin{itemize}
	\item Denote the local coordinates on $M$ by $$x := (x^0,x^{1},x^{2},x^{3}) = (x^{0},x^{\prime}).$$  
	\item For any $p,q \in M$,   $p < q$ says that there exists a future-pointing causal curve from $p$ to $q$, and $p\le q$ means that either $p<q$ or $p=q$ holds. 
	\item For any $p,q \in M$,    $p \ll q$ says that there exists a future-pointing time-like curve from $p$ to $q$.
	\item For any $p\in M$ and $A\subseteq M$,   the causal future  and causal past of $p$ and $A$ are defined respectively to be
\begin{align*}
 J^{+}(p) &:= \{q\in M; p\le q\},\\ J^{-}(p) &:= \{q\in M; q\le p\}, \\ J^{\pm}(A) &:= \bigcup_{p\in A} J^{\pm}(p).
\end{align*}
\item For any $p\in M$ and $A\subseteq M$,  the chronological future and chronological past of $p$ and $A$ are  defined respectively to be
\begin{align*}
  I^{+}(p) &=\{q\in M; p \ll q\},\\ I^{-}(p) &= \{q\in M; q\ll p\},\\   I^{\pm}(A) &:= \bigcup_{p\in A}I^{\pm}(p).
\end{align*}
 
\end{itemize}

In the sense of \cite{Bernal-Sanchez-2007}, Lorentzian manifold $(M,g)$ is said to be \textit{globally hyperbolic} if
\begin{itemize}
     \item no closed causal curves exist on $M$;
     \item for all $p,q\in M$, the causal diamond \begin{align}\label{eqn: causal diamond} \mathbb{D} := J^{+}(p) \cap J^{-}(q)\end{align} is compact.
\end{itemize}
Such condition implies that $(M, g)$ admits a global $1+3$ splitting in ‘time’ and ‘space’. Namely, $(M,g)$ is isometric to $\R\times M'$ furnished with metric 
$$
      -\beta(x^{0},x^{\prime}) dx^0 \otimes dx^0 + g'(x^{0},x^{\prime}, dx'),\quad \text{for all  }x^{0}\in \R,x^{\prime}\in M',
$$
where $\beta$ is a smooth positive function and $g'$ is a Riemannian metric on $M'$ smoothly dependent on $x^0$. For each fixed $x^0$, the hypersurface $\{x^0\}\times M'$ is a Cauchy hypersurface which any inextendible causal curve intersects at most once. The initial value of a wave equation is usually defined on such Cauchy hypersurfaces.

\subsection{The wave equation and local measurements}
The D'Alembertian operator $\Box_g$ is defined by 
\[ (\Box_g u, v)_{L^2(M)} := (\nabla_g u, \nabla_g v)_{L^2(M)}, \quad \forall u, v \in C^\infty_c(M).\]  In local coordinates $(x^0, x^1, x^2, x^3)$, it takes the form
$$
\Box_g  = -|\det g|^{-1/2}\partial_{x^j}\l(|\det g|g^{jk}\partial_{x^k}\r),
$$
 where $(g^{jk})$ denotes the inverse matrix of $(g_{jk})$.

 Taking advantage of global hyperbolicity, we are allowed to endow \eqref{eqn : semilinear wave} with an initial value explicitly as follows. If $T > 0$, the underlying Cauchy problem, in the domain $(0, T) \times M' \subset M$, is  
\begin{equation} \label{semi_linear_wave}
  \begin{aligned}
    \begin{cases}
 \Box_g u(x)  + N(x,u(x)) = f(x),\quad & x\in (0,T)\times M';\\
   u(0,x') = 0,\partial_{x^0}u(0,x') = 0,\quad & x'\in M'.
  \end{cases}
      \end{aligned}
\end{equation}

Assume we can arbitrarily make the sources and measure the solutions of \eqref{semi_linear_wave} in a small open ball  $B  \subset  (  M',g'(0,\cdot))$ within time $T$. Here $|B|$ can be arbitrarily small. From the perspective of timespace, this amounts to measuring the source-to-solution type data of \eqref{semi_linear_wave}  in the region  
\begin{equation}\label{eqn: observation region mho}
	\rotom := (0,T)\times B \subset M.
\end{equation}

The forward theory (e.g. \cite[Appendix III, Theorem 3.7 and 4.2]{Choquet-Bruhat}, \cite[Theorem III]{Hughes-Kato-Marsden-1976}, \cite{K}, \cite[Theorem 6.3.1]{Rauch-book}) yields the local measurements \eqref{semi_linear_wave}  in $(0, T) \times M'$.  The sources useful for recovery lie in the set
$$
 \mathscr{C}^k_{\mho, \varepsilon} := \left\{ h\in H^{k}(M);k\ge4 ,\supp h \subseteq \mho,  \Vert h\Vert_{H^{k}(M)}\le \varepsilon\right\},
$$
where $\varepsilon$ is a sufficiently small constant depending on $(M,g)$ and $T$. Then, for any $f\in \mathscr{C}^k_{\mho, \varepsilon}$ the Cauchy problem (\ref{semi_linear_wave}) admits a unique solution
$$
   u \in  H^{k+1}((0,T)\times M'). 
$$
It follows a well-defined source-to-solution map 
  \begin{align}\label{eqn : StS}
  L_{N}:  \mathscr{C}^k_{\mho, \varepsilon} &\longrightarrow H^{k+1}(\rotom) ,\\ 
   f &\longmapsto u|_{\rotom}. \notag
  \end{align}

 It is worth pointing out that the data encoded in the source-to-solution map \eqref{eqn : StS} is fewer than the full Dirichlet-to-Neumann type boundary data. If the underlying domain is geodesically convex, any pair of a source point and its receiver point can be connected by a geodesic in the full boundary measurements. However, this is not applicable in our model with measurement \eqref{eqn : StS} as the size $|B|$ of the observation region can be arbitrarily small.

\subsection{Unique inversion of nonlinearity} The precise statement of the main result is the following.

\begin{theorem}\label{thm: main theorem}
  Fix a Lorentzian manifold $(M, g)$ and the measurement region $\mho$ as in \eqref{eqn: observation region mho}. 
  Assume that there are no cut points  along any null geodesic within the causal diamond 	$$\D := J^{+}(\rotom) \cap J^{-}(\rotom).$$ Let $N^{(1)}$ and $N^{(2)}$ be two functions  which both satisfy conditions (i)-(iv) and agree in $\mho$.
     If their source-to-solution maps $L_{N^{(j)}}$ for $j = 1, 2$, defined by \eqref{eqn : StS}, agree in $\mho$, then they must be the same pointwisely in  	$\D $. That is, 
  \[   L_{N^{(1)}} = L_{N^{(2)}} \quad \Longrightarrow \quad   N^{(1)} = N^{(2)} \, \mbox{pointwisely in $\D$}.  \]

\end{theorem}

\begin{remark}The causal diamond $\D$ is the maximal area where one can detect $N$ from the measurements in $\mho$. This is a consequence of the finite speed of propagation of waves.\end{remark}

\begin{remark}
The constrains (i)-(iv) on $N$ seem to be inevitable for the framework of higher order linearization. 
\begin{itemize}
	\item Such methodology is based on the microlocal analysis of distorted plane waves and semilinear wave equations. Therefore, the coefficients have to be smooth.
	\item If $N(x, u)$ is linear in $u$ at some $x$, then no nonlinear scattering will occur at $x$ such that no information from $x$ will be received. So we are only able to reconstruct $N(x, u)$ which is nonlinear in $u$ everywhere in the model.  
	\item The wave equations from physical models are often the Euler-Lagrange equations derived through calculus of variation. Hence it is legitimate to assume that $N(x, u)$ is homogeneous in $u$. 
	\item In the higher order linearization scheme, nonlinear waves are approximated by multiple linear waves. 	It is thus necessary to assume that $N(x, u)$ has product features in $u$. Here comes the analyticity in $u$.
\end{itemize} 
\end{remark}

\subsection{Methodologies} 

The classical tools for inverse coefficient problems of wave equations are the boundary control method \cite{Belishev,BK} and   the unique continuation principle \cite{T}. However, it does not, in general, apply to wave equations with time-dependent coefficients, as Alinhac's counterexample \cite{A} suggested. 

For nonlinear wave equations with time-dependent coefficients,   \cite{KLU} developed the framework of higher order linearization. The general philosophy is to exploit the nonlinear scattering of the underlying equation, through  higher order linearized equations, and then  the coefficients are reconstructed by the data of the solution operators of these linearized equations. 

Specifically, \cite{KLU} introduced the four wave linearization to solve the inverse problem of quadratic wave equations \eqref{eqn : quadratic wave equation}. Four linear distorted plane waves are designed to interact in the region to be retrieved, and then the interaction prompts a scattered wave returning back to the receiver. Since the dimension of the Lorentzian manifold is $4$, four linearly independent waves can deliver the scattered waves along any light-like direction. 

The process of linearization can be simplified for wave equations with cubic nonlinearity. \cite{CLOP} observed that three linear distorted plane waves are sufficient to produce such a returning wave. Consequently, the three wave linearization enables us to solve inverse coefficient problems of wave equations with cubic nonlinearity, as the cubic structure naturally accommodates three wave interactions. Along this line, \cite{FO, FLO} used three Gaussian beams, instead of distorted plane waves, to detect the second order and zeroth order terms. However, this simplified linearization method breaks down for quadratic equations \eqref{eqn : quadratic wave equation}, due to the structure of null vectors.

For wave equations with $m$-th degree nonlinearity, it is natural to perform a linearization scheme with $m$ waves. \cite{LUW} considered wave equations with special nonlinearity \eqref{eqn : higher degree nonlinearity without potential}.  The linearization approach with $m$ distorted plane waves was employed to reconstruct the $m$-th coefficient $h_m(x)$ in \eqref{eqn : higher degree nonlinearity without potential}.  Then  \cite{Uhlmann-Zhang}  generalized this result to some quasi-linear cases. However, this approach does not seem to work when the first degree term $h_1$ is present. On the other hand, \cite{Hintz-Uhlmann-Zhai-IMRN} used $m$ Gaussian beam linearization to recover \eqref{eqn : higher degree nonlinearity without potential}. Nonetheless, this strategy is unable to deliver the second degree term $h_2$. Recently, \cite{Lassas-Liimatainen-Pohjola-Tyni-2025} considered the inverse source problem of nonlinear hyperbolic equations, together with the recovery of zero-th order terms and quadratic nonlinear terms.

Furthermore, \cite{Lin-Liu-Liu-2024,Qiu-Xu-Ye-Zhou-2025} investigated \eqref{Q : N} with nonlinearity of the form \eqref{eqn : higher degree nonlinearity without potential} on Minkowski space-time. They proved the uniqueness of recovery from the full Dirichlet-to-Neumann type data. In this case, the recovery of $h_1$ is from the light-ray transform and the recovery of $h_2$ is from two-fold linearization. However,  this framework   does not apply to Theorem \ref{thm: main theorem}. On one hand, the invertibility of the light-ray transform on a globally hyperbolic Lorentzian manifold is unknown.  On the other hand,   the  source point and and its   receiver point are not connected by a light-like geodesic in the model of Theorem \ref{thm: main theorem} as we mentioned above. In contrast, we have to construct suitable broken light-like geodesics to connect them, and recover $h_1$ via the corresponding (truncated) broken light-ray transforms as in \cite{CLOP, FO}.

In summary, all of the existing linearization schemes have limitations such that it only recovers the coefficients of some specific forms or from the full boundary measurements on Minkowski space-time.

To overcome these challenges,  we propose a hybrid linearization program with both distorted plane waves and Gaussian beams. The specific steps are the following.

\subsubsection*{Reduction to nonlinearity of power series type}
	
 Making use of the analyticity of $N(x, u)$ in $u$, we expand $N(x, u)$ as a power series in $u$ near $0$, i.e. \[N(x, u) = h_1(x) u + h_2(x) u^2 + \cdots + h_m(x) u^m + \cdots\] when $u$ lies in a small disc $\mathcal{B} \subset \mathbb{C}$ centred at $0$. Then we shall first consider \begin{align}
		\label{eqn : semilinear wave with power series} \Box_g u + \sum_{k = 1}^\infty h_k(x) u^k = f,
	\end{align} and show the following simpler version of the main theorem.

  	\begin{theorem}\label{thm : simplified theorem}
		Let $(M, g)$, $\D$ and $\mho$ be as in Theorem \ref{thm: main theorem}. 	
		Suppose, in addition,  that there are  $h_m^{(j)} \in C^\infty(\D)$ for $j = 1, 2$ and $m \in \Z_+$ such that $h_m^{(1)} = h_m^{(2)}$ in $\mho$ and \[N^{(j)}(x, u) = h_1^{(j)}(x) u + h_2^{(j)}(x) u^2 + \cdots + h_m^{(j)}(x) u^m + \cdots, \] and   $L_{N^{(j)}}$ are defined by  \eqref{eqn : StS} accordingly. Then  there holds
		\[   L_{N^{(1)}} = L_{N^{(2)}} \quad \Longrightarrow \quad   N^{(1)} = N^{(2)} \quad \mbox{pointwisely in $\D$}.  \]
		
	\end{theorem}

	\subsubsection*{The $3$-wave linearization }
	
  We perform   the linearization on \eqref{eqn : semilinear wave with power series} with $3$ distorted plane waves which delivers the linearized system \eqref{one-fold linearization}-\eqref{three-fold linearization}. On one hand, the information of the leading terms recovers $h_3$ in $\D$ as well as $h_2^2$ away from $\supp h_3$.  On the other hand, the data of the subleading terms  determines $h_1$ on  $\supp h_3$  and  $\supp h_2 \setminus \supp h_3$ respectively. Compared with the analogous analysis for the flat case in \cite{Chen-Lu-Zhang-2025}, the curved geometry of $(M, g)$ makes the three-fold linearized waves no longer distorted plane waves but Lagrangian distributions with non-trivial Keller--Maslov bundle. As such, we remedy this issue by analysing the principal symbols with non-trivial Keller--Maslov bundle.     
 See Proposition \ref{prop : 3 wave}.
 
\subsubsection*{The $m$-wave linearization}
For $y \in  \D \setminus  (\supp h_2 \cup \supp h_3)$, assume $h_m(y) \neq 0$ and $h_j(y) = 0$ for $j = 2, \cdots, m-1$.  To reconstruct the coefficients at $y$, we perform linearization with  $3$ distorted plane waves and $m-3$ Gaussian beams. This gives a system in the form of \eqref{eqn: l-fold linearization}.
Analogous with $3$-wave linearization, the leading terms of \eqref{m-fold linearization} likewise yield $h_m$ in $\D \setminus   (\supp h_2 \cup \supp h_3)$, while the subleading terms determine $h_1$ on  $\supp h_m\setminus   (\supp h_2 \cup \supp h_3)$.    If we go through all $y \in  \D \setminus  (\supp h_2 \cup \supp h_3)$, $h_1$ will be fully obtained.
See Proposition \ref{prop: uniqueness of $V$ via higher order nonlinearity}.	
 
\subsubsection*{The $4$-wave linearization}

 Next, we perform  the linearization on \eqref{eqn : semilinear wave with power series} with $4$ distorted plane waves, which, along with the knowledge of $h_1$, leads to the complete retrieval of $h_4$ and $h_2$ respectively. See Proposition \ref{prop: recovery of h_2 and h_4}.

\subsubsection*{The $m$-wave linearization (revisited)}
  Moreover, we perform  again the linearization on \eqref{eqn : semilinear wave with power series} with $3$ distorted plane waves and $m-3$ Gaussian beams, which, together with the knowledge of $h_1$, concludes the reconstruction of $h_m$ for $m > 4$.  See Proposition \ref{prop: recovery of higher order nonlinearities}.

\subsubsection*{Analytic continuation}The propositions in Steps 2-5 together prove Theorem \ref{thm : simplified theorem}. In the end, analytic continuation extends $N(x, \cdot)$ from $\mathcal{B}$ to $\mathbb{R}$, which completes the proof of Theorem \ref{thm: main theorem}.

\bigskip

In summary, we lay out the following roadmap for the inversion of $N(x,u)$.

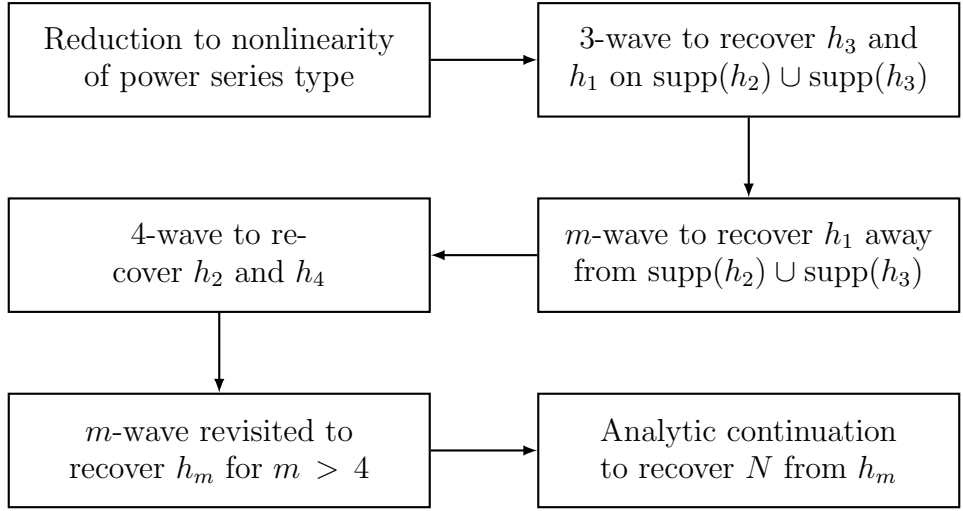
\begin{figure}[htbp]
    % 预定义开始结束圆角矩形的样式
    \tikzstyle{startstop} = [rectangle, rounded corners=0.5cm, thick, minimum width = 3cm, minimum height=1cm,text centered, draw = black, ]
    % 预定义流程步骤方框的样式
    \tikzstyle{process} = [rectangle, thick, minimum width=5cm, minimum height=1.5cm, text centered, text width = 5cm, inner sep = 8pt, draw = black,]
    % 预定义连线样式
    \tikzstyle{arrow} = [thick,->,>=latex]
    \centering
    \begin{tikzpicture}[node distance=3cm]
        \node (m1) [process] {Reduction to nonlinearity of power series type};
        \node (m2) [process, right = 40pt of m1] {3-wave to recover $h_3$ and $h_1$ on $\supp(h_2) \cup\supp(h_3)$};
        \node (m3) [process, below = 30pt of m1] {4-wave to recover $h_2$ and $h_4$};
        \node (m4) [process, below = 30pt of m2] {$m$-wave to recover $h_1$ away from $\supp(h_2) \cup\supp(h_3)$};
        \node (m5) [process, below = 30pt of m3] {$m$-wave revisited to recover $h_m$ for $m>4$};
        \node (m6) [process, below = 30pt of m4] {Analytic continuation to recover $N$ from $h_m$};
%4-wave to recover $h_2$ and $h_4$
      \draw[arrow](m1) -- (m2);
      \draw[arrow](m2) -- (m4);
      \draw[arrow](m4) -- (m3);
      \draw[arrow](m3) -- (m5);
      \draw[arrow](m5) -- (m6);
    \end{tikzpicture}
\end{figure}

One might be puzzled why we mix up the inversion of all $h_m$, rather than reconstruct $h_m$ in an orderly fashion as it goes in \cite{LUW, Hintz-Uhlmann-Zhai-IMRN, Uhlmann-Zhang}. 
 In fact, the complication to determine the nonlinearity $N(x, u)$ in \eqref{eqn : semilinear wave} is two-fold. 
\begin{itemize}

	\item The presence of $h_1$  severely entangles the structure of the product terms in the linearized equations, no matter what strategy is adopted. Therefore, it significantly complicates the singularity analysis. In particular, $h_1$ appears in the linear wave equation, which means it is present in any linearized equation. As such, one cannot fully determine $h_2$ and $h_k$ for $k > 3$ without the knowledge of $h_1$.  
	
	\item For each $k \in \Z_+$, $h_k$ might have zeros such that the principal singularities of the scattered waves emanating from different points might result from the derivatives of $N$ in $u$ of distinct orders. Consequently, a much finer discussion on $\supp h_k$ for each $k \in \N$ is required.  
\end{itemize}

\subsection{Structure of the paper}This paper is organized as follows. First of all, Section \ref{sec: linear waves with potential} provides a brief review of useful tools of analysis and geometry. Following the roadmap, the proof of Theorem \ref{thm : simplified theorem} is carried out from Section \ref{sec : 3 wave linearization} to Section \ref{sec : Step 5}.

\section{Preliminaries}\label{sec: linear waves with potential}
%The framework of higher order linearization utilizes linear waves to approach underlying nonlinear waves.
%The relevant linear waves to approximate nonlinear waves in \eqref{eqn : semilinear wave} or \eqref{eqn : semilinear wave with power series} are solutions to
%\begin{equation}\label{linear waves with a potential}
 % \Box_g u + h_1 u = f
%\end{equation}
%where $h_1$ is a smooth function.

%We first review the useful tools in analysis and geometry.

\subsection{Lagrangian distributions}  
%A class of special solutions that play a central role in inverse problems for nonlinear hyperbolic equations is distorted plane waves, which are modeled by a class of conormal distributions. We first recall the conormal distributions in \cite[Section 18]{H3}.

Unlike the  Minkowski case, the linearized waves  turns out to be  no longer distorted plane waves but Lagrangian distributions. Thus, we have to invoke  H\"ormander's framework of Lagrangian distributions in \cite{H}.

Let $X$ be a smooth $n$-manifold with local coordinates $(x^1,\dots,x^n)$ and denote by $(\xi_1,\dots,\xi_n)$ the corresponding coordinates on the cotangent space $T_{x}^{\ast}X$. Let $\Lambda\subseteq T^{\ast}X \b 0$ be a smooth conic Lagrangian submanifold (see e.g. \cite[Definition 21.1.8, 21.2.5]{H3}).

Suppose $\phi(x,\theta)$ with $(x,\theta) \in U\times \R^{N}$ is a nondegenerate phase function defined in \cite[Definition 21.2.15]{H3}, where $U$ is a coordinate neighborhood in $X$. Then, $\phi$ parametrizes $\Lambda$ in a conic neighborhood $\Gamma \subseteq U \times \R^{N}$ in the sense  
 \begin{equation}\label{eqn: nondegenerate phase function}
    \Lambda = \{(x,d_{x}\phi(x,\theta)) \in T^\ast U : d_{\theta}\phi(x,\theta) = 0, (x,\theta)\in\Gamma\}.
\end{equation}

 Lagrangian distributions are defined as follows.

\begin{definition}
Suppose that $\Lambda$ is a conic Lagrangian submanifold of $T^{\ast}X\b 0$. The space $I^{m}(\Lambda)$ consists of distributions $u = \sum_{j\in J} u_j\in \mathcal{D}^{\prime}(X)$ such that the supports of $u_j$ are locally finite and the expression of $u_j$ is
$$
 u_{j}(x)= (2\pi)^{-\frac{n+2N_j}{4}}\int_{\R^{N}}e^{\imath(\phi_{j}(x,\theta) -\pi N_j/4)}a_{j}(x,\theta)d\theta,
$$
where 
\begin{itemize}
    \item $\phi_{j}(x,\theta)$ is a nondegenerate phase function parametrizing $\Lambda$ as in \eqref{eqn: nondegenerate phase function};

    \item $a_{j}(x,\theta)$ is a classical symbol of order $m+n/4 - N_j/2$ supported in a conic neighborhood in $U_j\times \R^{N_j}$.
\end{itemize}
\end{definition}

For a Lagrangian distribution $u \in I^{m}(\Lambda)$,  its principal symbol $\sigma[u]$ of $u$ on $\Lambda$ can be intrinsically defined as  a section lying in the quotient space
\begin{equation}\label{eqn : principal symbol with HD KM bundle}
 \sigma[u] \in S^{m+\frac{n}{4}}(\Lambda,\Omega_{1/2}\otimes L) / S^{m+\frac{n}{4}-1}(\Lambda,\Omega_{1/2}\otimes L),
\end{equation}
where $\Omega_{1/2}$ is the half-density bundle over $\Lambda$ and $L$ is the Keller--Maslov line bundle. More precisely, an element $\sigma[u] \in S^{m+n/4}(\Lambda,\Omega_{1/2}\otimes L)$ means that
$$
\sigma[u]_{j} = \imath^{\sigma_{jk}} \sigma[u]_{k}\quad \text{in }U_{j}\cap U_{k},
$$
where $\sigma[u]_{j}$ is the local expression of $\sigma[u]$ in $U_{j}$, and
\begin{equation}\label{eqn: definition of sigma_{jk}}
\sigma_{jk} := ((\sgn \partial_{\theta}^{2}\phi_{k}(x,\theta_k) - N_k) - (\sgn\partial_{\theta}^{2}\phi_{j}(x,\theta_j)-N_j))/2.
\end{equation}

A special class of Lagrangian distributions are conormal distributions / distorted planes waves, which are associated with conormal bundles to $p$-submanifolds.
A $p$-submanifold $Y\subseteq X$ of codimension $n-r$ is locally furnished with coordinates $(x_1,\cdots,x_r)$, for some $0\le r \le n-1$. Then the conormal bundle $N^{\ast}Y \subseteq T^{\ast}X$ is locally given by
$$
  N^{\ast}Y = \{(x^1,\dots,x^r,\xi_{r+1},\dots,\xi_n)\}.
 $$
 For brevity we write 
 \begin{align}\label{eqn : '-'' coordinates} \left\{ \begin{aligned}
x &= (x^{\prime}, x^{\prime\prime}),  &&\mbox{for $x^{\prime} = (x^1,\dots,x^r)$ and $x^{\prime\prime} = (x^{r+1},\dots,x^n)$}; 
\\ 
\xi &= (\xi^{\prime},\xi^{\prime\prime}), &&\mbox{for $\xi^{\prime} = (\xi_1,\dots,\xi_r)$ and $\xi^{\prime\prime} = (\xi_{r+1},\dots,\xi_n)$}. 
\end{aligned} \right.\end{align}

 \begin{definition}

 We say that a distribution $u \in \mathcal{D}'(X)$ is conormal to $Y$ and of order $m$ if $u$, in coordinates \eqref{eqn : '-'' coordinates}, takes the local form
  $$
     u(x) = \int_{\R^{n-r}} e^{\imath x^{\prime\prime}\cdot\xi^{\prime\prime}}a(x^{\prime},\xi^{\prime\prime})d\xi^{\prime\prime}, \quad \text{where} \quad a\in S^{m+n/4-(n-r)/2}\l(\R^{r}_{x^{\prime}}\times\mathbb{R}^{n-r}_{\xi^{\prime\prime}}\r).
  $$
The principal symbol of $u$ is given by 
$$
  \sigma[u] = (2\pi)^{\frac{n}{4}-\frac{n-r}{2}}a(x^{\prime},\xi^{\prime\prime})|dx^{\prime}|^{\frac{1}{2}}|d\xi^{\prime\prime}|^{\frac{1}{2}} \in S^{m+\frac{n}{4}}(N^{\ast}Y)/S^{m+\frac{n}{4} - 1}(N^{\ast}Y).
$$
We  denote by $I^{m}(N^{\ast}Y)$ the collection of  conormal distributions of order $m$ to $Y$. \end{definition}
 Note that the Keller--Maslov bundle associated with a conormal bundle is trivial.

%To study linear wave equations with a source, Melrose-Uhlmann \cite{MU} and Guillemin-Uhlmann \cite{GU} developed the theory of paired conormal distributions. 

The paired Lagrangian distributions, introduced by Melrose--Uhlmann \cite{MU}, are useful to describe waves near the source as well as the products of conormal waves.

Let $\Lambda_0$ be a conic Lagrangian submanifold of $T^{\ast}X\b0$ and $\Lambda_1$ be a conic Lagrangian submanifold of $T^{\ast}X\b 0$ with boundary. We say $(\Lambda_0,\Lambda_1)$ form an intersecting pair of Lagrangian submanifolds if $\Lambda_0$ and $\Lambda_1$ intersect cleanly. That is, $\Lambda_0 \cap \Lambda_1 = \p \Lambda_1$ and $$T_\lambda (\Lambda_0) \cap T_\lambda (\Lambda_1) = T_{\lambda} (\p \Lambda_1), \quad \mbox{for all $\lambda \in \p \Lambda_1$}.$$

An intersecting paired Lagrangian distribution is defined as follows. 

\begin{definition}
    Let $(\Lambda_0,\Lambda_1)$ be an intersecting pair of Lagrangian submanifolds of $T^\ast X \b 0$. The space $I^{m}(\Lambda_0,\Lambda_1)$ consists of $u\in \mathcal{D}^{\prime}(X)$ such that
    $$
       u = u_0 + u_1 +\sum_{j}F_j v_j,
    $$
    where
    \begin{itemize}
        \item $u_0\in I^{m-1/2}(\Lambda_0)$ and $u_1 \in I^{m}(\Lambda_1\b\partial\Lambda_1)$;
        \item $F_j$ is a zeroth order elliptic Fourier integral operator associated with a homogeneous canonical transformation $\chi_j$ mapping $V_j\subseteq \partial\Lambda_1$ to $T^{\ast}\R^{n}$;
        \item $v_j$  %introduced in \cite[Definition 2.1]{MU},
         is  microlocally supported near $\Lambda_0 \cap \Lambda_1$, and locally takes the form \[v_j(x) = \int_0^\infty \int_{\R^n} \exp\l(\imath (x_1 - s) \xi_1 + \sum_{j=2}^n x_j \xi_j  \r) a_j(s, x, \xi)\,d\xi ds\] with $a_j \in S_{\mathrm{cl}}^{m+1/2-n/4}(T^\ast \R^n)$;
        \item the summation $\sum_{j}F_j v_j$ is locally finite.
    \end{itemize}
\end{definition}

Let $\chi_0,\chi_1$ be a zero-th order pseudodifferential operator satisfying 
$$
 \WF(\chi_0)\cap \Lambda_1 = \varnothing\quad\text{and}\quad\WF( \chi_1) \cap \Lambda_0 = \varnothing.
$$
By \cite[Proposition 4.1]{MU}, we have
$$
 \chi_0 u \in I^{m-1/2}(\Lambda_0\b \partial\Lambda_1),\quad \chi_1 u \in I^{m}(\Lambda_1\b\partial\Lambda_1).
$$
The principal symbol of $u$ away from $\partial\Lambda_1$ is 
\begin{equation}
    \begin{aligned}
        \begin{cases}
            \sigma[u](\lambda_0) = \sigma[\chi_0 u](\lambda_0)/\sigma[\chi_0](\lambda_0),\quad \lambda_0\in\Lambda_0\b\partial\Lambda_1,\\
            \sigma[u](\lambda_1) = \sigma[\chi_1 u](\lambda_1)/\sigma[\chi_1](\lambda_1),\quad \lambda_1\in\Lambda_1\b\partial\Lambda_1.
            \nonumber
        \end{cases}
    \end{aligned}
\end{equation}
On the boundary $\partial\Lambda_1 = \Lambda_0\cap\Lambda_1$, the symbol transition map $\mathscr{R}$ is invariantly defined in \cite[(4.7),(4.12)]{MU} such that
\begin{align*}
  \mathscr{R}: S^{m-\frac{1}{2}+\frac{n}{4}}(\Lambda_0\b\partial\Lambda_1) &\longrightarrow S^{m+\frac{n}{4}}(\Lambda_1\b\partial\Lambda_1)\\
  \sigma[u]|_{\Lambda_0\b\partial\Lambda_1} &\longmapsto  \sigma[u]|_{\partial\Lambda_1}.
\end{align*}

Moreover, \cite[Lemma 1.1]{GU} computed the product of two distributions conormal to two transverse $p$-submanifolds.
\begin{proposition}\label{prop: Greenleaf-Uhlmann}
  Let $Y_{(1)}$ and $Y_{(2)}$ be two transversally intersecting  $p$-submanifolds of $X$ and $u_{(j)} \in I^{\mu_j}\l( N^{\ast}Y_{(j)}\r)$ for $j=1,2$. Then, microlocally away from $N^{\ast}Y_{(1)} \cup N^{\ast}Y_{(2)}$, 
$$ 
  u_{(1)}u_{(2)}\in I^{\mu_1+\mu_2+1}\l( N^{\ast}\l(Y_{(1)}\cap Y_{(2)} \r)\r).
$$
Furthermore, the principal symbol of $u_{(1)}u_{(2)}$ at $(x,\xi)\in N^{\ast}\l(Y_{(1)}\cap Y_{(2)} \r) \setminus \cup_{j=1}^{2}N^{\ast}Y_{(j)}$ is given by
$$
  \sigma\l[u_{(1)}u_{(2)}\r](x,\xi) = \sigma\l[u_{(1)}\r]\l(x,\xi_{(1)}\r) \sigma\l[u_{(2)}\r]\l(x,\xi_{(2)}\r),
$$
where $\xi_{(j)}\in N^{\ast}Y_{(j)}\b 0$ and $\xi = \xi_{(1)} + \xi_{(2)}$.
\end{proposition}

\subsection{Distorted plane waves}\label{sec: Distorted plane waves}  
  In this section, we  follow  \cite[Section 3.2.3]{KLU} to construct distorted plane waves solving the inhomogeneous linear wave equation
  \begin{equation}\label{eqn: linear wave equation}
  \Box_g u+ h_1u =f. 
  \end{equation}

Denote by $LM\subseteq TM$ the null tangent bundle, 
$$
  LM := \{(p,\zeta)\in TM: g(\zeta,\zeta) = 0\}.
$$
For $x_0 \in \rotom$, $\zeta_0 \in L_{x_0}M$ and sufficiently small $s_0>0$,  let $\gamma_{x_0,\zeta_0}(s)$ be the light-like geodesic originating from
 $$
   \gamma_{x_0,\zeta_0}(0) = x_0,\quad \dot{\gamma}_{x_0,\zeta_0}(0) = \zeta_0.
 $$
Consider the following neighbourhood $\W_{x_0,\zeta_0,s_0}$ of $\zeta_0$ on a sphere in $L_{x_0} M$, 
$$
  \W_{x_0,\zeta_0,s_0} =\l\{\eta\in L_{x_0}M: \|\eta-\zeta_0 \|_{\R^n}<s_0, \| \eta\|_{\R^n} = \|\zeta_0 \|_{\R^n} \r\},
  $$
  where $\|\cdot\|_{\R^n}$ is the usual vector norm in $\R^n$.
The source $f$ we favour for $\eqref{eqn: linear wave equation}$ has a wavefront set 
$$
 \Sigma(x_0,\zeta_0,s_0) := \l\{(x_0,r\eta^{\flat})\in T^{\ast}M: \eta \in \W_{x_0,\zeta_0,s_0}, r\in \R\b 0  \r\}.
$$
%Let $\W_{x_0,\zeta_0,s_0} = \V_{x_0,\zeta_0,s_0} \cap LM$, and define
Then the waves $u$,  emanating from such a source $f$, have the following singular support and wavefront set  \begin{equation}
  \begin{aligned}
     K(x_0,\zeta_0,s_0) &:=\l\{\gamma_{x_0,\eta}(s)\in M : \eta\in \W_{x_0,\zeta_0,s_0}, s\in \R^{+} \r\},\\
     \Lambda(x_0,\zeta_0,s_0) &:= \l\{ \l(\gamma_{x_0,\eta}(s),r\dot{\gamma}_{x_0,\eta}(s)^{\flat} \r) \in T^{\ast}M: \eta\in \W_{x_0,\zeta_0,s_0},s\in \R^{+},r\in \R\b0\r\}.
   \nonumber
   \end{aligned}
\end{equation}
In particular, $\Lambda$ intersects $\Sigma(x_0,\zeta_0,s_0)$ transversally at the boundary $\p \Lambda (x_0,\zeta_0,s_0)$,
$$
  \partial\Lambda(x_0,\zeta_0,s_0) = \overline{\Lambda}(x_0,\zeta_0,s_0)   \cap \Sigma(x_0,\zeta_0,s_0).
$$

%Let $Q_{g,V} = (\Box_g+V)^{-1}$ denote the causal inverse operator, i.e., the forward fundamental solution of the linear wave \eqref{linear waves with a potential}, see \cite[Theorem 3.2.11]{BGP}. A geometric representation of its Schwartz kernel is given by \cite{KR}. For convenience, denote
%$$
%Q := Q_{g,V}.
%$$ 
%Then, the solution $u$ to \ref{linear waves with a potential} with source $f$ can be written as
%

Let $Q$ be the solution operator to $\eqref{eqn: linear wave equation}$. Namely, the solution $u$ is 
\begin{equation}\label{eqn: causal inverse}
	  u =Q(f).
	\end{equation}	
By \cite[Proposition 6.6]{MU}, the Schwartz kernel $Q(x,y)$ of $Q$ lies in the space
$$
 I^{-\frac{3}{2}}(N^{\ast}\Delta,\Lambda_g),
$$ where $N^{\ast}\Delta$ is the conormal bundle of the diagonal $\Delta = {(x,x)\in M\times M}$ and $\Lambda_g$ is the Lagrangian submanifold associated with the bicharacteristic relation. %\cite[Proposition 26.1.3]{H4}.

Applying \cite[Proposition 6.6]{MU}, we can construct a solution to \eqref{eqn: linear wave equation} with a point source as follows.
\begin{proposition}\label{prop: linear wave}
Let $\Lambda =  \Lambda(x_0,\zeta_0,s_0)$ and $\Sigma =\Sigma(x_0,\zeta_0,s_0)$. Suppose $f$ belongs to $I^{\mu+3/2}(\Sigma)$. Then, the solution $u = Q(f)$ is a paired conormal distribution satisfying
$$
   u\in I^{\mu}( \Sigma,\Lambda), \quad \text{and} \quad u|_{M\b \{x_0\}} \in I^{\mu}(M\b \{x_0\}, \Lambda ).
$$
If $(y,\eta)\in \Lambda$ and $(x,\xi)\in \Sigma$ lie on the same bicharacteristics, then the principal symbols $\sigma[u]$ and $\sigma[f]$ satisfy the relation
\begin{equation}\label{eqn: causal inverse as a linear operator}
 \sigma[u](y,\eta) = \sigma\l[Q\r](y,\eta,x,\xi)\sigma[f](x,\xi),
\end{equation}
 Moreover, $\sigma[u]$ solves the following transport equation:
\begin{equation}\label{eqn: transport equation for linear waves}
 \begin{aligned} 
   \begin{cases}
     \mathscr{L}_{H_{\Box_g}}\sigma[u]=0,\quad &\text{on }\Lambda\b\partial\Lambda,\\
     \sigma[u]  = \mathscr{R}\l( \sigma\l[\Box_g\r]^{-1}\sigma[f]\r)\quad &\text{on }\partial\Lambda,
   \end{cases}
 \end{aligned}
\end{equation}
where $\mathscr{L}_{H_{\Box_g}}$ denotes the Lie derivative along the Hamiltonian vector field $H_{\Box_g}$.

\end{proposition}

\begin{remark}\label{eqn: principal symbol calculus in different languages}
 By \cite[(2.8)]{Chen-Lu-Zhang-2025}, the principal symbol $u$ at $(y,\eta)$ can be written as
\begin{equation}\label{eqn: fundamental solution of the principal symbol}
  \sigma[u](y,\eta) = \alpha(y,\eta)\sigma[u](x,\xi),
\end{equation}
where $\alpha$ is a zero-th order positive factor only depending on $\Lambda$ and $\Sigma$. Moreover, by the transport equation \eqref{eqn: transport equation for linear waves},
$$
  \sigma[u](x,\xi) = \mathscr{R}\l( \sigma\l[\Box_g\r]^{-1}\sigma[f]\r)(x,\xi),
$$
which shows that $\sigma[u]$ is independent of $h_1$.

  Combining \eqref{eqn: causal inverse as a linear operator} and \eqref{eqn: fundamental solution of the principal symbol}, we obtain the identity
 $$
     \alpha(y,\eta)\mathscr{R}\l( \sigma\l[\Box_g\r]^{-1}\sigma[f]\r)(x,\xi) = \sigma\l[Q\r](y,\eta,x,\xi)\sigma[f](x,\xi),
 $$
 which provides a relation between the principal symbols of $Q,u$ and $f$.
\end{remark}

The principal symbol calculus above are irrelevant to $h_1$. To extract the information of $h_1$, we consider the lower order term $u^{\rem}$ in the full distorted plane waves. Denote by $u^{\fre}$ the free wave solving
 \begin{equation}\label{eqn: equation of the free wave} 
   \Box_g u^{\fre} = f.
 \end{equation}
 By Proposition \ref{prop: linear wave}, we have
 \begin{equation}\label{eqn: microlocal property of free linear waves}
 u^{\fre} \in I^{\mu}(\Sigma,\Lambda)\quad\text{and}\quad \sigma\l[u^{\fre}\r] = \sigma[u].
 \end{equation}
Define the remainder wave by
 \begin{equation}\label{eqn: definition of the remiander wave}
    u^{\rem} := u - u^{\fre}.
 \end{equation}
 Applying $\Box_g+h_1$ to both sides yields
  \begin{equation}\label{eqn: equation of the remiander wave}
   (\Box_g+h_1)u^{\rem} = (\Box_g+h_1)u - \Box_g u^{\fre} - h_1u^{\fre} = -h_1u^{\fre}.
 \end{equation}
\cite[Proposition 2.2]{Chen-Lu-Zhang-2025} computed the principal symbol of $u^{\rem}$ as 
 \begin{equation}\label{eqn: lower order symbols}
\sigma\l[u^{\rem}\r](y,\eta) = -\imath\sigma[u](x,\xi)\int_{0}^{s}h_1(\gamma(\tau))d\tau  \mod S^{\mu-1}(\Lambda\b\partial\Lambda).
 \end{equation}

\subsection{Gaussian beams}The Gaussian beam solutions to hyperbolic equations was originally introduced by Babich--Ulin \cite{BU}, and further developed by Ralston \cite{R}.
We next construct Gaussian beam solutions on the globally hyperbolic Lorentzian manifold $(M,g)$, following the presentation in \cite[Section 3]{FIKO} and \cite[Section 4]{FO}.  The construction is based on the following Fermi coordinates near light-like geodesics.
\begin{lemma}\label{lemma : fermi} Let $\gamma: (a-\rho, b+\rho) \rightarrow M$ be a light-like geodesic segment on $M$, where $\rho > 0$ and $a < b$. 

\begin{itemize} \item[(i)] There exists a coordinate chart $(U,\Phi)$ of $\gamma([a,b])$ with local coordinates denoted by $(y^0, y^1, y^2, y^3) := (s,y^{\prime})$ such that
	\begin{align*}
		\Phi(U)=(a-\rho^{\prime},b+\rho^{\prime}) \times B(0,\rho^{\prime})  ,\quad\text{and} \quad\Phi(\gamma(s))=(s,0,0,0),
	\end{align*}
    where $B(0, \rho^\prime)$ is a sufficiently small open ball in $\R^{3}$.
\item[(ii)]	The metric tensor $g$ restricted on $\gamma$ satisfies
	\begin{align*}
		 g|_{\gamma} &= 2ds\otimes dy^1 + dy^{2}\otimes dy^{2} + dy^{3}\otimes dy^{3} 
	\end{align*} 
    and $ \frac{\partial}{\partial y^i} g_{jk}|_{\gamma} =0$ for $i,j,k =0,1,2,3$.
\end{itemize}\end{lemma}

%Fix a null geodesic segment $\gamma$ and local Fermi coordinates as above. Denote a tubular set \[\mathcal{V} := \left\{ (s, y') : s\in (a-\rho',b+\rho'), |z'| < \rho' \right\}\]

 Consider the following ansatz supported on the tubular neighborhood  $\mathcal{V} := \overline{\Phi(U)}$  
\begin{align} \label{eqn : wkb ansatz}
	\mathcal{U}_{\lambda}(y)  := e^{\imath \lambda \phi(y)}A_{\lambda}(y),\quad \text{for }\lambda >0.
\end{align} Here the phase function $\phi \in C^\infty(\mathcal{V})$ is of the form \begin{equation}\label{eqn: construction of Gaussian beams1}
			\phi(s,y^{\prime}) = \sum_{j=0}^N \phi_{j}(s,y^{\prime})
	\end{equation} such that
\begin{itemize}
	\item each $\phi_j$ is a complex-valued homogeneous polynomial  in $y^{\prime}$ of degree $j$; 
	\item the imaginary part of $\phi$ vanishes on $\gamma$, i.e. $\Im \phi |_\gamma = 0$;
	\item the real part of $e^{\imath \lambda \phi}$ decays like a Gaussian function, i.e.  $\Im \phi (y) |_{\mathcal{V}} \gtrsim |y'|^2$.
\end{itemize}
 In the meanwhile, the amplitude function $A_{\lambda}\in C_c^\infty(\mathcal{V})$ takes the form 
	\begin{equation}\label{eqn: construction of Gaussian beams2}
		\begin{aligned} 
			A_{\lambda}(s,y^{\prime}) &= \chi \l(\frac{|y^{\prime}|}{\rho^{\prime}}\r) \sum_{k=0}^{N}\lambda^{-k} a_{k}(s,y^{\prime}),\quad
			a_{k}(s,y^{\prime}) = \sum_{j=0}^{N} a_{k,j}(s,y^{\prime}),
		\end{aligned}
	\end{equation}
	where \begin{itemize} \item $a_{k,j}$ are complex-valued homogeneous polynomials of degree $j$ in $y^{\prime}$; \item $\chi: \mathbb{R}\to [0,+\infty)$ is a cut-off function such that $\chi(t)\equiv 1$ for $|t|\le 1/4$ and $\chi(t)\equiv 0$ for $|t|\ge 1/2$. \end{itemize}

The standard WKB method enables us to determine the phase and amplitude. First, applying $\Box_g + h_1$ to \eqref{eqn : wkb ansatz} gives the following expansion in $\lambda$
\[  \left(\Box_g + h_1\right) \mathcal{U}_{\lambda} = e^{\imath \lambda \phi} \left( \lambda^2 (\mathcal{H} \phi) A_\lambda    -\imath \lambda \mathcal{T} A_\lambda + \left(\Box_g + h_1\right) A_\lambda \right),\]
where 
$$
	\mathcal{H} \phi := \langle d\phi, d\phi \rangle_g,  \quad \text{and}\quad \mathcal{T} A := 2 \langle d\phi, d A \rangle_g -  \left( \Box_g  \phi \right) A.
$$ 
To make the leading order term vanish, $\phi$ must solve the eikonal equation
\begin{align} \label{eqn : eikonal}
	\p_y^\alpha \left( \mathcal{H} \phi \right) (s, 0, 0, 0) = 0, \quad \text{for all } s \in ( a , b ).
\end{align} From subsequent terms, the amplitudes must solve the iterative transport equations 
\begin{align} \label{eqn : transport} \left\{\begin{aligned}
	\p_y^\alpha \left(\mathcal{T} a_0\right) (s, 0, 0, 0) &= 0, &  \text{for all } s \in ( a, b), \\
	\p_y^\alpha \left(- \imath \mathcal{T} a_k + \left(\Box_g + h_1\right) a_{k-1}\right) (s, 0, 0, 0) &= 0,  &  \text{for all } s \in ( a , b ),\end{aligned} \quad  \right.
\end{align}  for any multi-index $\alpha =(\alpha_1,\alpha_2,\alpha_3)\in \N^3$ with $|\alpha| := |\alpha_1|+|\alpha_2|+|\alpha_3| \leq N$.

%To recover the coefficients $h$ and $h_1$ in the semi-linear model \eqref{semi_linear_wave} via higer order nonlinearities, we recall the explicit expressions for $\phi$ and amplitude terms $a_{0,0},a_{1,0}$. 

According to \cite{FO}, we have the following solutions to \eqref{eqn : eikonal} and \eqref{eqn : transport} 
\begin{align}\label{eqn: the expression for phase phi and amptilute a_0, a_1}
 &\phi_{0}(s,y^{\prime} ) = 0,\quad \phi_{1}(s,y^{\prime}) = y_1, \quad \phi_{2}(s,y^{\prime}) = \sum_{i,j=1}^{3}H_{ij}(s)y^{i}y^{j},\\
 &a_{0,0}(s) =(\det Y(s))^{-\frac{1}{2}},\quad a_{1,0}(s) = b_{1,0}(s) + c_{1,0}(s),
\end{align}
where
\begin{equation}\label{eqn: Gaussian beam the expression for the subleading term}
 \begin{aligned}
   b_{1,0}(s) &= -\frac{\imath}{2}(\det Y(s))^{-\frac{1}{2}} \int_{0}^{s}(\Box_g a_{0})(\tilde{s},0)(\det Y(\tilde{s}))^{\frac{1}{2}} d\tilde{s},\\
   c_{1,0}(s) &= -\frac{\imath}{2}(\det Y(s))^{-\frac{1}{2}} \int_{0}^{s}h_1(\tilde{s},0) d\tilde{s}.
 \end{aligned}
\end{equation}
 Fix $s_0 \in (a,b)$. The matrix $H$ is a complex-valued symmetric matrix, solving the following Riccati equation for $s\in (a,b)$
\begin{equation}\label{eqn: Riccati}
 \frac{d}{ds} H+ HCH+ D = 0, \quad H(s_0) = H_0,\quad \Im H_{0}>0,
\end{equation}
where $C$ and $D$ are matrices with $D_{ij}= \partial^{2}_{ij} g^{11}/4$ and $C_{jj} = 2$ for $j=2,3$ and $C_{i,j} =0$ otherwise. The solvability of the Riccati equation was discussed in \cite[Section 8]{KKL}. Moreover, $Y(s)$ is a nondegenerate matrix on $(a,b)$ satisfying
  \begin{equation}\label{eqn: relation between Y and H}
  \det(\Im H(s))\cdot |\det Y(s)|^2 = \det(\Im(H_0)).
  \end{equation}

One can construct a source $f_{\lambda, \c{x},\c{\xi}} $ supported near $\c{x}$ to produce the desired Gaussian beam ansatz $\U_{\lambda,\c{x},\c{\xi}}$, which is supported near the light-like geodesic $\gamma_{\c{x},\c{\xi}}: (a,b)\rightarrow M$, where $a<0<b$. We parametrize the geodesic as
$$
\gamma(0) = \c{x} \in \rotom,\quad\gamma(s_{\into}) = \c{y}\quad \text{and} \quad\dot{\gamma}^{\flat}(0) = \c{\xi}.
$$
 Choose small $\rho^{\prime} > 0 $ and cut-off functions $\zeta_{\pm}\in C^{\infty}(\mathbb{R})$ such that
\begin{equation}
	\begin{aligned}
		\zeta_{-}(x_{0}) &=  \begin{cases}
			0,\quad x_{0}\le \c{x}_{0}-{\rho^{\prime}}\\
			1,\quad x_{0}\ge \c{x}_{0}-{\rho^{\prime}}/{2}
		\end{cases}
		\quad\text{and}\quad
		\zeta_{+}(x_{0}) &=  \begin{cases}
			0,\quad x_{0}\ge \c{x}_{0}\\
			1,\quad x_{0}\le \c{x}_{0}-{\rho^{\prime}}/{2}
		\end{cases}.
		\nonumber
	\end{aligned}
\end{equation}
Take the source to be
\begin{equation}\label{eqn: construction of source of the Gaussian beam}
	f_{\lambda, \c{x},\c{\xi}} := \zeta_{+}(\Box_g+h_1)\l(\zeta_{-} \mathcal{U}_{\lambda, \c{x},\c{\xi}}\r)\in C_{c}^{\infty}(\rotom).
\end{equation}
By \cite[(27)]{FO}, the linear wave $u_{\lambda,\c{x},\c{\xi}}$ generated by the source $f_{\lambda,\c{x},\c{\xi}}$ obeys 
 \begin{equation}\label{eqn: approximate property of the Gaussian beam}
  \l\|u_{\lambda,\c{x},\c{\xi}} - \zeta_{+}\mathcal{U}_{\lambda,\c{x},\c{\xi}} \r\|_{C((0,T)\times M^{\prime})} \lesssim \lambda^{-4}.
 \end{equation}

%We express the Taylor expansion of $u_{\lambda}(\c{y})$ in terms of $\lambda$ as follows. 
From equations \eqref{eqn : wkb ansatz}-\eqref{eqn : transport} and \eqref{eqn: approximate property of the Gaussian beam}, for $\c{y} = (s_{\into},0)$ in Fermi coordinate by Lemma \ref{lemma : fermi}, it follows that
$$
u_{\lambda,\c{x},\c{\xi}}(\c{y}) = \U_{\lambda,\c{x},\c{\xi}}(\c{y})+O(\lambda^{-4}) = \sum_{k=0}^{1} a_{k,0}(s_{\into},0)\lambda^{-k} + O(\lambda^{-2}).
$$
According to \eqref{eqn: the expression for phase phi and amptilute a_0, a_1} and \eqref{eqn: Gaussian beam the expression for the subleading term},  $a_{0,0}$ and $a_{1,0}$ are given by
$$
  a_{0,0}(s_{\into}) = (\det Y(s_{\into}))^{-\frac{1}{2}}, \quad\text{and}\quad a_{1,0}(s_{\into},0) = b_{1,0}(s_{\into},0)+c_{1,0}(s_{\into},0).
$$
By \cite[Lemma 5.3]{FLO}, one can choose $Y$ and $H$ in \eqref{eqn: relation between Y and H} such that $a_{0,0}(s_{\into}) = 1$. Then, 
\begin{align}\label{eqn: expansion of Gaussian beams at y}
 u_{\lambda,\c{x},\c{\xi}}(\c{y}) &= a_{0,0}(\c{y}) + \lambda^{-1} a_{1,0}(\c{y})+ O(\lambda^{-2})\\
 &=1+\lambda^{-1}\l( b_{1,0}(\c{y})-\frac{\imath}{2}\int_{0}^{s_{\into}}h_1(\gamma(s))ds\r) + O(\lambda^{-2}), \notag
\end{align}
where $b_{1,0}$ is independent of $h_1$.

\subsection{Broken geodesics in Lorentzian geometry}
% The information is carried along broken light-like geodesics emanating from $\rotom$ and back to $\rotom$. 

As in \cite{CLOP, FO}, the pointwise information of lower order coefficients at every $\mathbf{y} \in \D \setminus \mho$ is encoded in some ray transforms over a broken geodesic segments turning at $\mathbf{y}$. Then the values of the underlying coefficients at $\mathbf{y}$ may be recovered from the measurements of the ray information within $\mho$. Such broken geodesics are illustrated in Figure \ref{fig:figure1}.

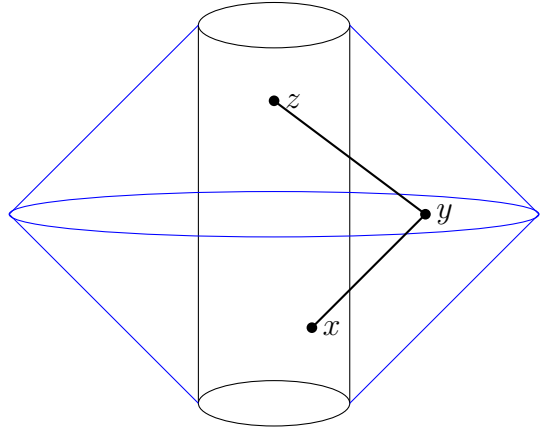
\begin{figure}[htbp]
	\centering
	\begin{minipage}{0.45\textwidth}
		\centering
		\begin{tikzpicture}[scale=1]
			\draw[black] (0,0) ellipse (1 and 0.3);
			% 画底面椭圆（虚线表示被遮挡部分）
			\draw[black] (0,-5) ellipse (1 and 0.3);
			% 画两条竖直边
			\draw[black] (-1,0) -- (-1,-5);
			\draw[black] (1,0) -- (1,-5);
			
			%画Diamond
			\draw[blue] (1,0) -- (3.5,-2.5);
			\draw[blue] (1,-5) -- (3.5,-2.5);
			\draw[blue] (-1,0) -- (-3.5,-2.5);
			\draw[blue] (-1,-5) -- (-3.5,-2.5);
			
			%画Diamond中间的椭圆显得有立体感
			\draw[blue]  (0,-2.5) ellipse (3.5 and 0.3);

			%画x,y,z
			
			\fill (0.5,-4) circle (2pt); % 半径 3pt 的实心点
			\node[right] at (0.5,-4) {$x$}; % 在右边标注 x
			
			\fill (0,-1) circle (2pt); % 半径 3pt 的实心点
			\node[right] at (0,-1) {$z$}; % 在右边标注 x
			
			\fill (2,-2.5) circle (2pt); % 半径 3pt 的实心点
			\node[right] at (2,-2.5) {$y$}; % 在右边标注 x
			
			%连成线
			
			\draw[thick] (0.5,-4) -- (2,-2.5);
			\draw[thick] (0,-1) -- (2,-2.5);
			
		\end{tikzpicture}

	\end{minipage}
	\caption{The broken light-like geodesic from $x$ via $y$ to $z$}
    \label{fig:figure1}
\end{figure}

The existence of broken geodesics for each $\mathbf{y} \in \D \setminus \mho$ is clear for Minkowski spacetime. However, this is non-trivial for globally hyperbolic Lorentzian manifolds. 

Denote a world line by $\mu_{a}(s) = (s,a)$. The observation region $\rotom$ in \eqref{eqn: observation region mho} can be foliated by such world lines:
$$
  \rotom = \bigcup_{a\in B(0,d)}\mu_{a}(0,T).
$$
Thus, the causal diamond $\D$ in \eqref{eqn: causal diamond} can be foliated by
$$
  \D = \bigcup_{a_1,a_2\in B(0,d)} J^{+}(\mu_{a_1}(0,T)) \cap J^{-}(\mu_{a_2}(0,T)).
$$
We introduce the notations
\begin{align*}
   \mathbb{L}^{-}(\rotom) &:= \{(x,y) \in \rotom\times (\D\b\rotom);\  y\in J^{+}(x)\b I^{+}(x) \},\\
   \mathbb{L}^{+}(\rotom) &:= \{(y,z) \in (\D\b\rotom)\times \rotom;\   z\in J^{+}(y)\b I^{+}(y) \},
\end{align*}
where $y\in J^{+}(x)\b I^{+}(x)$ means that there is a future-pointing light-like geodesic from $x$ to $y$ and there is no future-pointing time-like geodesic from $x$ to $y$. 
%The time separation function $\tau(p,q)$ for $p<q$ is defined as the supremum of $$ \tau(p,q) := \sup_{\alpha}\int_{I}\sqrt{-g(\dot{\alpha}(s),\dot{\alpha}(s))}ds$$ where $\alpha: I\to M$ is a piecewise smooth causal curve  from $p$ to $q$.

   \cite[Lemma 4]{FO}  confirms the existence of broken geodesics.

\begin{lemma}\label{lemma: broken light-like geodesics}  
 For all $y\in \D\b\rotom$, there exists $z\in \rotom$ such that $(y,z)\in \mathbb{L}^{+}(\rotom)$. Moreover, for all $(y,z) \in \mathbb{L}^{+}(\rotom)$, there is $x\in \rotom$ such that $(x,y)\in \mathbb{L}^{-}(\rotom)$. The light-like geodesic $\gamma_1$ connecting $x$ to $y$ and the light-like geodesic $\gamma_{2}$ connecting $y$ to $z$ intersect only once.
\end{lemma}

\section{The $3$-wave linearization}\label{sec : 3 wave linearization}
%\label{sec: Recovery of the potential via cubic nolinearities}

In this section, we shall prove  {the uniqueness of $h_3$ and the uniqueness of $h_1$ on the support of $h_3$ and $h_2$ as follows.}
\begin{proposition}\label{prop : 3 wave}
	Let $(M, g)$, $\D$, $\mho$, $N^{(j)}$, $h_m^{(j)}$ and $L_{N^{(j)}}$ be as in Theorem \ref{thm : simplified theorem}. If $L_{N^{(1)}} = L_{N^{(2)}}$, then we have
	\begin{align}
		\label{h_3 is known}	h_3^{(1)} &= h_3^{(2)} && \mbox{in $\D$},\\ 
		\label{eqn: uniquness of the square of h_2}		\left|h_2^{(1)}\right| &= \left|h_2^{(2)}\right| && \mbox{away from $\supp h_3^{(j)}$}.\end{align} 
	Moreover, we have
	\begin{align}	\label{eqn: Recovery of the potential via cubic nonliearities}
		h_1^{(1)} &= h_1^{(2)}  && \mbox{on $\supp h_3^{(j)}$},\\
		\label{eqn: Recovery of the potential via quadratic nonliearities}
	h_1^{(1)}  &= h_1^{(2)}   &&\mbox{on $\supp(h_2)\b\supp(h_3)$}. 
	\end{align} 
	
\end{proposition}

The strategy of the proof is to perform on \eqref{eqn : semilinear wave with power series} the linearization with $3$ distorted plane waves, and then conduct the principal and subprincipal symbol calculi. The symbol calculi lead to the uniqueness in Proposition \ref{prop : 3 wave}. This strategy was used  in \cite{Chen-Lu-Zhang-2025} for the flat spacetime. Nonetheless, the curved geometry of $(M, g)$ in this paper complicates the argument by generating non-trivial Keller--Maslov indices of principal symbols.

\subsection{The $3$-wave interactions}\label{sec: three wave interactions} 
To begin with, we briefly review the rationale of the $3$-wave linearization on globally hyperbolic Lorentzian manifolds. The reader is referred to \cite{CLOP, FO, Chen-Lu-Zhang-2025} for more details. %The process of three-wave interactions is discussed in  \cite[Section 3]{}.

Fix a point $\c{y}\in \D\b\rotom$. Lemma \ref{lemma: broken light-like geodesics} guarantees that there exists a point $\c{z}\in \rotom$ and a future-pointing light-like geodesic $\gamma_{\out}(s)$ such that for some $s_{\out}>0$
\begin{equation}\label{eqn: light-like geodesic back to mho}
\gamma_{\out}(0) = \c{y},\quad \gamma_{\out}(s_{\out}) = \c{z}.
\end{equation}
The bicharacteristic $\beta_{\out}(s)$ associated with $\gamma_{\out}$ is given by
\begin{equation}\label{eqn: bicharacteristics}
 \beta_{\out}(s) = \l(\gamma_{\out}(s),\dot{\gamma}^{\flat}_{\out}(s) \r).
\end{equation} where $\cdot^\flat$ stands for the musical isomorphism from tangent vectors to cotangent vectors. In addition, we denote \[  \c{\eta} := \dot{\gamma}_{\out}^{\flat}(0),\quad \c{\zeta} := \dot{\gamma}^{\flat}_{\out}(s_{\out}).\]
Furthermore, Lemma \ref{lemma: broken light-like geodesics} yields a point $\c{x}\in \rotom$ and a future-pointing light-like geodesic $\gamma_{\into}(s)$ such that for some $s_{\into}>0$
$$
 \gamma_{\into}(0) = \c{x}, \quad \gamma_{\into}(s_{\into}) = \c{y}, \quad\text{and}\quad \c{\nu} := \dot{\gamma}_{\into}^{\flat}(s_{\into}) \neq \c{\eta}.
$$

In a small neighbourhood of $\c{y}$, we adopt local coordinates such that
$$
  \c{\nu} = (-1,1,0,0),\quad\text{and}\quad \c{\eta} = (-1,\cos\theta,\sin\theta,0),\quad \mbox{for some $\theta\in(0,2\pi)$}.
$$
  For small $r>0$, we take small perturbations of $\c{\nu}$ 
$$
  \c{\nu}_{(2)} := (-1,\sqrt{1-r^{2}},r,0),\quad\text{and}\quad   \c{\nu}_{(3)} := (-1,\sqrt{1-r^{2}},-r,0).
$$
Since each $\c{\nu}_{(j)}$ is light-like, there exist a point $\c{x}_{(j)}\in \rotom$ and a light-like geodesic $\gamma_{(j)}$   such that
$$
   \gamma_{(j)}(0) = \c{x}_{(j)},\quad \gamma_{(j)}(s_{\into}) = \c{y},\quad \dot{\gamma}_{(j)}^{\flat}(s_{\into}) = \c{\nu}_{(j)}.
$$
If we flip the sign of $\c{\nu}$ and write
\begin{equation}\label{eqn: sign of the light-like covector}
\c{\xi}_{(1)} := -\c{\nu}, \quad  \c{\xi}_{(2)} := \c{\nu}_{(2)},    \quad\c{\xi}_{(3)} := \c{\nu}_{(3)},
\end{equation}
\cite[Lemma 1]{CLOP} shows that there exist $\{\kappa_{(j)} : j = 1, 2, 3\}$  such that
\begin{equation}\label{eqn: linear relation}
 \c{\eta} = r^{-2}\underbrace{(2 b(\theta)+ O(r))}_{{\kappa}_{(1)}}\c{\xi}_{(1)}+r^{-2}\underbrace{(b(\theta)+ O(r))}_{{\kappa}_{(2)}}\c{\xi}_{(2)}+r^{-2}\underbrace{(b(\theta)+ O(r))}_{{\kappa}_{(3)}}\c{\xi}_{(3)},
\end{equation}
where $b(\theta) := 1-\cos\theta$.
For convenience, we use the shorthand notation
\begin{equation}\label{eqn: notation of the linear relation}
  \c{\eta}_{(j)} := r^{-2}\kappa_{(j)}\c{\xi}_{(j)}.
\end{equation}

The useful inputs for inversion of coefficients  in \eqref{eqn : semilinear wave with power series} are of the form 
\begin{align} \varepsilon_{(1)}   f_{(1)} +  \varepsilon_{(2)}   f_{(2)} + \varepsilon_{(3)}  f_{(3)}  \label{eqn : 3 conormal source}\end{align}
where each $\varepsilon_{(j)} > 0$ is a sufficiently small number and each $ f_{(j)}\in H^{s}_{\comp}(\rotom)$  satisfies
\begin{itemize}
\item $f_{(1)}\in I^{\mu+\frac{3}{2}}(\Sigma(\c{x}_{(1)},-\dot{\gamma}_{(1)}(0),s_0))$;
 
 \item For $j=2,3$,
 $f_{(j)}\in I^{\mu+\frac{3}{2}}(\Sigma(\c{x}_{(j)},\dot{\gamma}_{(j)}(0),s_0))$;

\item These sources are causally independent, i.e. 
$$
\supp\l(f_{(j)}\r) \cap J^{+}\l(\supp\l(f_{(k)}\r)\r)  =\varnothing\quad\text{for all } 1\le j\neq k\le 3.
$$ 
 \item The principal symbol of each $f_{(j)}$ is positively homogeneous of degree $\mu+5/2$.
\end{itemize}

 Let $u_{(j)}, j= 1, 2, 3$ be the solution to the linear wave equation \eqref{eqn: linear wave equation} with source $f_{(j)}$.
From the discussion in Section \ref{sec: Distorted plane waves}, we know that
  \begin{itemize}
 	\item each $u_{(j)}$ lies in $I^{\mu}\l(N^{\ast}\{\c{x}_{(j)}\} , N^{\ast}K_{(j)} \r)$, where $K_{(j)}$ takes the form 
 	\begin{align*}
 		K_{(1)} &:= K_{(1)}\l(\c{x}_{(1)},-\dot{\gamma}_{(1)}(0),s_0 \r),\\  K_{(j)} &:= K_{(j)}\l(\c{x}_{(j)},\dot{\gamma}_{(j)}(0),s_0 \r),\quad j=2,3;
 	\end{align*}
 	\item the principal symbol of $u_{(j)}$, restricted to $N^{\ast}K_{(j)} \b N^{\ast}\{\c{x}_{(j)}\}$, is positively homogeneous of degree $\mu+1$;
 	\item $u_{(j)}$ admits the following decomposition  as in \eqref{eqn: equation of the free wave} and \eqref{eqn: definition of the remiander wave}
 	\begin{equation*}%\label{eqn: decomposition of u_j}
 		u_{(j)} = u_{(j)}^{\fre} + u_{(j)}^{\rem};
 	\end{equation*}
 	%where $u_{(j)}^{\fre}$ is independent of the unknown coefficients to be recovered.
 	\item the remainder $u_{(j)}^{\rem}$ belongs to
$ I^{\mu-1}\l(N^{\ast}\{\c{x}_{(j)}\} , N^{\ast}K_{(j)} \r);
$
 	\item the principal symbol of $u_{(j)}^{\rem}$, restricted to $N^{\ast}K_{(j)} \b N^{\ast}\{\c{x}_{(j)}\}$, is positively homogeneous of degree $\mu$;

\item furthermore,  the homogeneity of the principal symbols of $u_{(j)}$ and $u_{(j)}^{\rem}$, together with \eqref{eqn: lower order symbols}, yields 
\begin{align}
\label{relation between the principal symbol of free wave and remiander wave by homogeneity}  \sigma\l[u_{(j)}^{\rem}\r](\c{y},r^{-2}\kappa_{(j)}\c{\xi}_{(j)}) &= \l(r^{-2}\kappa_{(j)}\r)^{\mu}\sigma\l[u_{(j)}^{\rem}\r]
	(\c{y},\c{\xi}_{(j)}) \\
	\notag  &=-\imath\l(r^{-2}\kappa_{(j)}\r)^{\mu}\sigma\l[u_{(j)}\r]
	(\c{y},\c{\xi}_{(j)})\int_{\gamma_{(j)}} h_1\\
	%\label{relation between the principal symbol of free wave and remiander wave by homogeneity}  
\notag	 &=-\imath\l(r^{-2}\kappa_{(j)}\r)^{-1}\sigma\l[u_{(j)}\r]
	(\c{y},r^{-2}\kappa_{(j)}\c{\xi}_{(j)})\int_{\gamma_{(j)}} h_1,
\end{align}
where we denote the truncated integrals of  $h_1$ along   $\gamma_{(j)}$ by
\begin{align*}
	\int_{\gamma_{(1)}}h_1 &:= \int_{0}^{-s_{\into}}h_1\l(\gamma_{(1)}(-s)\r)ds,\\ 
	\int_{\gamma_{(j)}}h_1 &:= \int_{0}^{s_{\into}}h_1\l(\gamma_{(j)}(s)\r)ds,\quad j=2,3.
\end{align*}

 \end{itemize}

% If $I,J$ are non-empty subsets of $\{1,2,3\}$ with $I\cap J = \emptyset$, then $\cap_{i\in I}K_{(i)}$ and $\cap_{j\in J}K_{(j)}$ intersect transversally.

%To exploit the $3$-wave interaction, we enter in  \eqref{eqn : semilinear wave with power series} the source
%$$
 % f = \e_{(1)}f_{(1)}+\e_{(2)}f_{(2)}+\e_{(3)}f_{(3)},
%$$
The solution $u(\e)$ to \eqref{eqn : semilinear wave with power series}, associated with source \eqref{eqn : 3 conormal source}, is smooth in $\e$.  The following derivatives of $u$ at $\e=0$  
\begin{align*}\begin{aligned}
   u_{(i)} &:= \partial_{\e_{(i)}}u(\e)|_{\e=0}, && i=1,2,3,\\
   u_{(jk)} &:=\partial_{\e_{(j)}\e_{(k)}}^{2}u(\e)|_{\e=0},&& \{j, k\} \subset \{1, 2, 3\},\\
   u_{(123)} &:= \partial_{\e_{(1)}\e_{(2)}\e_{(3)}}^{3}u(\e)|_{\e=0},&&
\end{aligned}\end{align*}
  solve the  linearized system
\begin{align}
   \label{one-fold linearization} \Box_g u_{(i)} + h_1 u_{(i)} &= f_{(i)},\\
   \label{two-fold linearization} \Box_g u_{(jk)}+ h_1 u_{(jk)} &= -2h_2 u_{(j)} u_{(k)},\\
   \label{three-fold linearization} \Box_g u_{(123)}+ h_1 u_{(123)} &= -h_2\sum_{ \{i,j,k\} = \{1,2,3\} } u_{(i)} u_{(jk)} -6h_3 u_{(1)} u_{(2)} u_{(3)}.
\end{align}

\subsection{Partial recovery of cubic and quadratic terms}
%In this section, we prove the uniqueness of the cubic nonlinearity $h_3$. Moreover, we show that $(h_{2})^{2}$ is uniquely determined away from the support of $h_3$, which will be useful in the recovery of $V$ on the support of $h_2$.
The inversion of $h_2$ and $h_3$ may be viewed as the inverse source problem of \eqref{three-fold linearization}.

 To solve this problem, we first conduct the symbol calculus of $u_{(123)}$ in \eqref{three-fold linearization}. By the principle of superposition, we decompose $u_{(123)}$ as
\begin{equation}\label{decomposition of three-fold linearization}
 u_{(123)} := \mathcal{V}_{(123)} + \mathcal{W}_{(123)},
\end{equation}
where $\mathcal{V}_{(123)}$ and $\mathcal{W}_{(123)}$ respectively solve  
\begin{align}
  \label{higher-order three-fold linearization} \l(\Box_g +h_1 \r)\mathcal{V}_{(123)} &= -6h_3 u_{(1)} u_{(2)} u_{(3)},\\
  \label{lower-order three-fold linearization} \l( \Box_g +h_1\r)\mathcal{W}_{(123)} & =  -\sum_{ \{i,j,k\} = \{1,2,3\} } h_2 u_{(i)} u_{(jk)}.
\end{align}

By the  wavefront set calculus for products of distributions,  e.g. \cite[(57)]{KLU} and
 \cite[Theorem 8.2.10]{H1},   $\WF (u_{(ij)}u_{(k)})$ and $\WF(u_{(1)}u_{(2)}u_{(3)})$ are contained in 
\begin{equation}\label{eqn: wavefront set of three-fold linearization}
  N^{\ast}\l( K_{(1)}\cap K_{(2)} \cap K_{(3)}\r) \cup \l( \bigcup_{1\le j < k \le 3} N^{\ast}\l(K_{(j)} \cap K_{(k)} \r)\r) \cup \l(\bigcup_{j=1}^{3} N^{\ast}K_{(j)}\r).
\end{equation}

Let $(a,b)$ be the maximal interval of existence for the light-like geodesic $\gamma_{\out}$ defined in \eqref{eqn: light-like geodesic back to mho}. Namely, $\{\gamma_{\out}(s) : s\in(a,b)\}$ form an inextendible geodesic. We denote the image of the associated bicharacteristics $\beta_{\out}$ in \eqref{eqn: bicharacteristics}  by
\begin{equation}\label{eqn: image of the bicharacteristics}
  \Gamma := \{\beta_{\out}(s) : s\in (a,b)\}.
\end{equation}
By the fact that the summation of two light-like covectors are not light-like, and
$$
  N^{\ast}_{x}\l(K_{(j)}\cap K_{(l)}\r) = N^{\ast}_{x} K_{(j)} \oplus N^{\ast}_{x}K_{(l)},\quad 1\le j \neq l\le 3,
$$
 the following properties hold:
\begin{align}
  \label{eqn: first wavefront set property} \Gamma &\cap N^{\ast}\l( K_{(1)} \cap K_{(2)}\cap K_{(3)}\r) = \{(\c{y},\c{\eta})\};\\
 \label{eqn: second wavefront set property}  \Gamma &\cap \l(\bigcup_{1\le j < k \le 3} N^{\ast}\l(K_{(j)} \cap K_{(k)} \r)\r)  = \emptyset,\quad\text{and}\quad \Gamma \cap  \l(\bigcup_{j=1}^{3} N^{\ast}K_{(j)}\r)= \emptyset.
\end{align}

Denote $\Lambda_{(123)} :=  N^{\ast}\l( K_{(1)}\cap K_{(2)} \cap K_{(3)}\r)$, and let $\Lambda_{(123)}^{g}$ be its future flowout. The following proposition, due to \cite[Proposition 3.2 and 3.3]{Hintz-Uhlmann-Zhai-IMRN} and \cite[Proposition 3.7]{LUW}, analyses the microlocal structures of $\mathcal{V}_{(123)}$ and $\mathcal{W}_{(123)}$.
\begin{proposition}\label{prop: principal symbol of three wave interaction}
Let $(\c{y},\c{\eta})$ and $(\c{z},\c{\zeta})$ be as in \eqref{eqn: light-like geodesic back to mho}.   Then   $\mathcal{V}_{(123)}$ and $\mathcal{W}_{(123)}$, microlocally restricted to $T^\ast M \setminus \cup_{j=1}^{3}N^{\ast}K_{(j)}$, are paired conormal distributions,
\begin{equation}\label{order of the two waves genereated by three-fold linearization}
  \mathcal{V}_{(123)} \in I^{3\mu+\frac{1}{2}}\l(\Lambda_{(123)},\Lambda_{(123)}^{g} \r) \quad\text{and}\quad \mathcal{W}_{(123)} \in I^{3\mu-\frac{3}{2}}\l(\Lambda_{(123)},\Lambda_{(123)}^{g} \r).
\end{equation}
Moreover, $ u_{(123)}$, microlocally restricted to $T^\ast M \setminus \cup_{j=1}^{3}N^{\ast}K_{(j)}$, obeys the following :
\begin{itemize} 
\item 
       If $h_3(\c{y}) \neq 0$, then   $ u_{(123)} \in I^{3\mu+1/2}\l( \Lambda_{(123)},\Lambda_{(123)}^{g}\r) $ with principal symbol 
       \begin{multline}\label{expression for the principal symbol of three-fold linearization}
        \sigma\l[ u_{(123)}\r]\l( \c{z},\c{\zeta}\r) = \sigma\l[\mathcal{V}_{(123)}\r](\c{z},\c{\zeta})\\
        =-6(2\pi)^{-2}\sigma\l[Q\r]\l(\c{z},\c{\zeta},\c{y},\c{\eta} \r) h_{3}(\c{y})\prod_{j=1}^{3}\sigma\l[u_{(j)}\r]\l(\c{y},r^{-2}\kappa_{(j)}\c{\xi}_{(j)} \r);
       \end{multline}
\item 
       If $h_3 \equiv 0$ in a neighbourhood of $\c{y}$, then   we have  in this neighbourhood that $ u_{(123)} \in I^{3\mu-3/2}\l( \Lambda_{(123)},\Lambda_{(123)}^{g}\r)$ with principal symbol
          \begin{multline}\label{expression for the lower order principal symbol of three-fold linearization}
          \sigma\l[ u_{(123)}\r](\c{z},\c{\zeta}) = \sigma\l[\mathcal{W}_{(123)} \r](\c{z},\c{\zeta})\\
          = 2(2\pi)^{-2}h_{2}^{2}(\c{y})\mathcal{G}_{1}(\c{y},\c{\eta}) \sigma\l[Q\r]\l(\c{z},\c{\zeta},\c{y},\c{\eta} \r)\prod_{j=1}^{3}\sigma\l[u_{(j)}\r]\l(\c{y},r^{-2}\kappa_{(j)}\c{\xi}_{(j)} \r),
         \end{multline}
       where $\mathcal{G}_{1}$ is a nonvanishing factor given by
       \begin{equation}\label{eqn: factor derived from the symbol of Q}
          \mathcal{G}_{1}(\c{y},\c{\eta}) := \sum_{\{i,j,k\} = \{1,2,3\} } \l(\sigma\l[ \Box_g\r]\l(\c{y}, r^{-2}\kappa_{(i)}\c{\xi}_{(i)} + r^{-2}\kappa_{(j)}\c{\xi}_{(j)} \r)\r)^{-1}.
       \end{equation}
       
\end{itemize}   
\end{proposition}

With Proposition \ref{prop: principal symbol of three wave interaction} in hand, we are ready to retrieve $h_3$ fully and $h_2$ partially.

\begin{proof}[Proof of \eqref{h_3 is known}-\eqref{eqn: uniquness of the square of h_2}]
The coincidence of the source-to-solution maps implies that
\begin{multline}\label{eqn: three-fold linearization with twice observation}
 u_{(123)}^{(1)} := \l.\partial_{\e_{(1)}}\partial_{\e_{(2)}}\partial_{\e_{(3)}}L_{h_1^{(1)},H^{(1)}}\l( \e_{(1)}f_{(1)}+\e_{(2)}f_{(2)}+\e_{(3)}f_{(3)}\r)\r|_{\epsilon = 0} \\
 =\l.\partial_{\e_{(1)}}\partial_{\e_{(2)}}\partial_{\e_{(3)}}L_{h_1^{(2)},H^{(2)}}\l( \e_{(1)}f_{(1)}+\e_{(2)}f_{(2)}+\e_{(3)}f_{(3)}\r)\r|_{\epsilon = 0} =: u_{(123)}^{(2)},
\end{multline}
where $\epsilon = (\e_{(1)},\e_{(2)},\e_{(3)})$. By equation \eqref{three-fold linearization} and Proposition \ref{prop: principal symbol of three wave interaction}, we have the equality of the principal symbols
\begin{equation}\label{eqn: equality of the principal symbol of u_123}
  \sigma\l[u_{(123)}^{(1)}\r](\c{z},\c{\zeta}) = \sigma\l[u_{(123)}^{(2)} \r](\c{z},\c{\zeta}),
\end{equation}

Since $\sigma[Q]$ and $\sigma\l[u_{(j)}\r]$ are non-vanishing and independent of $h_1$, the expression for the principal symbol from \eqref{expression for the principal symbol of three-fold linearization} gives \eqref{h_3 is known}.
Furthermore,  \eqref{expression for the lower order principal symbol of three-fold linearization} and \eqref{eqn: equality of the principal symbol of u_123}, along with $\mathcal{G}_1(\c{y},\c{\eta})\neq 0$, imply \eqref{eqn: uniquness of the square of h_2}.
\end{proof}

\subsection{Partial recovery of the potential via cubic nonlinearity}
Now we establish \eqref{eqn: Recovery of the potential via cubic nonliearities}, the partial uniqueness of the potential $h_1$ on the support of $h_3$. Since $h_3$ has been recovered on $\D$ in \eqref{h_3 is known}, we may write $h_3 := h^{(1)}_3 = h^{(2)}_{3}$.
%\begin{proposition}\label{prop: uniqueness of $V$ on the support of h_3}
%    Given the condition of Theorem \ref{thm : simplified theorem}, 
 %   $$
%        h_1^{(1)}(\c{y}) = h_1^{(2)}(\c{y}),\quad \text{for all }\c{y} \in \D \cap \supp(h_3).
 %   $$
%\end{proposition}

Since the potential term only shows up in lower order terms, the principal symbol information in Proposition \ref{prop: principal symbol of three wave interaction} is insufficient for inversion. We use the lower order symbol calculus   to recover the potential.

   We rewrite the RHS of \eqref{higher-order three-fold linearization} as 
$$
  f_{(123)} := -6h_3 u_{(1)}u_{(2)}u_{(3)}= -6 h_3 \prod_{j=1}^{3} \l( u_{(j)}^{\fre} + u_{(j)}^{\rem}\r), 
$$
where the free waves $u_{(j)}^{\fre}$ and the remainder waves $u_{(j)}^{\rem}$ are defined respectively in \eqref{eqn: equation of the free wave} and \eqref{eqn: definition of the remiander wave}. Decompose $f_{(123)}$ into a free source and a remainder source
\begin{equation}\label{eqn: the decomposition of f123fre and f123rem}
  f_{(123)} = f_{(123)}^{\fre} + f_{(123)}^{\rem}\quad\text{with}\quad
   f_{(123)}^{\fre} := -6 h_3 u_{(1)}^{\fre} u_{(2)}^{\fre}u_{(3)}^{\fre}.
\end{equation}
Let $\mathcal{V}_{(123)}^{\fre}$ be the free waves, solving 
\begin{equation}\label{eqn: free wave of the higher order part for three-fold linearization}
  \Box_g \mathcal{V}_{(123)}^{\fre} =  f_{(123)}^{\fre}.
\end{equation}
Then   the remainder wave  $
  \mathcal{V}_{(123)}^{\rem} := \mathcal{V}_{(123)} - \mathcal{V}_{(123)}^{\fre} 
 $
 solves 
\begin{equation}\label{eqn: remainder wave of the higher order part for three-fold linearization}
  \l(\Box_g+h_1 \r)\mathcal{V}_{(123)}^{\rem} = f_{(123)}^{\rem} - h_1\mathcal{V}_{(123)}^{\fre}.
\end{equation}
Adapting the strategy of the Minkowski case from \cite[Lemma 3.2]{Chen-Lu-Zhang-2025} for Lorentzian manifolds, we have

\begin{lemma}\label{lemma: microlocal structure of the higher order remiander wave of three fold-linearization}
Let $\chi$ be a microlocal cut-off near $(\c{z},\c{\zeta})$. Then there holds
\begin{equation}\label{eqn: order of the higher-order remiander wave}
  \chi\l(\mathcal{V}_{(123)}^{\rem} \r) \in I^{3\mu-\frac{1}{2}}\l(\Lambda_{(123)}^{g} \r),
\end{equation}
and for $r\to 0+$  
\begin{equation}\label{eqn: principal symbol higher-order remiander wave}
  \sigma\l[\mathcal{V}_{(123)}^{\rem} \r](\c{z},\c{\zeta})  = -\imath r^{-6(\mu+1)}\mathcal{A}_{(123)}h_3(\c{y})\l(\int_{0}^{s_{\out}}h_1\l(\gamma_{\out}(s)\r)ds + O(r^{2}) \r),
\end{equation}
where $\mathcal{A}_{(123)}$ is a nonvanishing factor independent of the coefficients to be recovered.
\end{lemma}

\begin{proof}
We first compute the principal symbol of $\mathcal{V}_{(123)}^{\fre}$ defined by \eqref{eqn: free wave of the higher order part for three-fold linearization}, which satisfies the following transport equation
 \begin{equation}\label{eqn:transport equation of free wave}
 \begin{aligned} 
   \begin{cases}
     \mathscr{L}_{H_{\Box_g}}\sigma\l[\mathcal{V}_{(123)}^{\fre}\r]=0\quad &\text{on }\Lambda_{(123)}^{g}\b\partial\Lambda_{(123)}^{g},\\
     \sigma\l[\mathcal{V}_{(123)}^{\fre}\textbf{}\r]  = \mathscr{R}\l( \sigma\l[\Box_g\r]^{-1}\sigma\l[f_{(123)}^{\fre}\r]\r)\quad &\text{on }\partial\Lambda_{(123)}^{g}.
   \end{cases}
 \end{aligned}
\end{equation}
By Proposition \ref{prop: Greenleaf-Uhlmann}, 
\begin{equation}\label{eqn: the order of the principal symbol of f123fre}
\sigma\l[f_{(123)}^{\fre}\r]\in S^{3\mu+3}(\Lambda_{(123)},\Omega_{1/2}\otimes L).
\end{equation}
Thus, 
$$
p:=\sigma\l[\mathcal{V}_{(123)}^{\fre}\r] \in S^{3\mu+3/2}\l(\Lambda_{(123)}^{g},\Omega_{1/2}\otimes L\r).
$$

Denote by $\beta_{\out}$ the bicharacteristic from $(\c{y},\c{\eta})$ to $(\c{z},\c{\zeta})$   defined in \eqref{eqn: bicharacteristics}. Choose a finite open covering such that
$$
    \beta_{\out}([0,s_{\out}]) \subseteq \bigcup_{j=1}^{n}U_j,
$$
where $\{U_j \subseteq T^{\ast}M\b0\}$ are open conic neighbourhoods such that
\begin{itemize}
    \item $(\c{y},\c{\eta}) \in U_1$ and $(\c{z},\c{\zeta})\in U_n$;
    \item $U_j$ only intersects $U_{j-1}$ and $U_{j+1}$ for $2\le j \le n-1$;
    \item there exists $0 =: s_0 < s_1<\cdots < s_n:= s_{\out}$ such that
    $\beta_{\out}([s_{j-1},s_{j}])\subseteq U_j$;
    \item the half-density bundle $\Omega_{1/2}$ and the Keller--Maslov bundle $L$ in \eqref{eqn : principal symbol with HD KM bundle} are locally trivial over each $U_j$. 
\end{itemize}

Let  $\omega_j$ be the half-density over $\Lambda_{(123)}^{g}\cap U_j$. The Keller--Maslov bundle on $\Lambda_{(123)}^{g}\cap U_j$ can be trivialized by $\imath^{m_j}$ for some $m_j\in \mathbb{Z}$. Then, we have a local trivialization of $p$ as
$$
p_j = \imath^{m_j}
(a_j \omega_j)|_{U_j\cap\Lambda_{(123)}^{g}},$$
where $a_j$ is a classical symbol on $U_j \cap \Lambda_{(123)}^{g}$.
By \eqref{eqn: definition of sigma_{jk}}, $p_j$ satisfies
\begin{equation}\label{eqn: transition map of Maslov bundle}
p_{j+1} = \imath^{\sigma_{j+1,j}}p_{j}.
\end{equation}

Restricting the transport equation \eqref{eqn:transport equation of free wave} to $\beta_{\out}([s_{j-1},s_j])$ yields
\begin{equation}\label{eqn: transport in equation on a local coordinate}
  \imath^{-m_j}\omega^{-1}_j \mathscr{L}_{H_{\Box_g}}p_j = H_{\Box_g}a_j + a_j \omega^{-1}_j\mathscr{L}_{H_{\Box_g}}\omega_j.
\end{equation}
Define the factor $\rho_j(s) = \int_{s_{j-1}}^{s}\omega_{j}^{-1}\mathscr{L}_{H_{\Box_g}}\omega_j(\beta_{\out}(s)) ds$ for $s\in (s_{j-1},s_j)$. Thus, \eqref{eqn: transport in equation on a local coordinate} reduces to
\begin{equation}\label{eqn: invariant}
 \partial_{s}\l( e^{\rho_j (s)}a_{j}(\beta_{\out}(s))\r) = 0,\quad s\in (s_{j-1},s_{j})
\end{equation}
We solve the transport equation with initial value $a_{j}(\beta_{\out}(s_{j-1}))$ as
$$
 \imath ^{-m_j}e^{\rho_j (s_j)}\l( \omega^{-1}_jp _j\r)(\beta_{\out}(s_j)) = \imath^{-m_j} \l( \omega^{-1}_jp_{j}\r)(\beta_{\out}(s_{j-1})).
$$
Plugging in \eqref{eqn: transition map of Maslov bundle}, we obtain
\begin{equation}\label{eqn: lagrangian submanifold 1}
p_{j+1}(\beta_{\out}(s_j)) = \imath^{\sigma_{j+1,j}}p_{j}(\beta_{\out}(s_j))= C_j\imath^{\sigma_{j+1,j}} p_{j}(\beta_{\out}(s_{j-1})),
\end{equation}
where $C_j$ is a nonvanishing factor for $j=1,2,\dots,n$ such that
\begin{equation}\label{eqn: lagrangian submanifold 2}
 C_j  = e^{-\rho_j(s_j)}\omega_{j}(\beta_{\out}(s_j))\omega_{j}^{-1}(\beta_{\out}(s_{j-1})),
\end{equation}
and $\imath^{\sigma_{j-1,j}}$ is the transition function of $L$ on $U_{j-1}\cap U_{j}$.

Iterating \eqref{eqn: lagrangian submanifold 1} and denoting 
\begin{equation}\label{eqn: patching lagrangian submanifold 1}
\alpha_{\out}(\c{z},\c{\zeta}) = \l(\imath^{\sum_{j=1}^{n-1}\sigma_{j+1,j}}\r) C_1\cdots C_n,
\end{equation}
 we have 
\begin{equation}\label{eqn: patching lagrangian submanifold 2}
\sigma\l[\mathcal{V}_{(123)}^{\fre}\r](\c{z},\c{\zeta}) = \alpha_{\out}(\c{z},\c{\zeta})\sigma\l[\mathcal{V}_{(123)}^{\fre}\r](\c{y},\c{\eta}).
\end{equation}

Next, we compute the principal symbol of $\mathcal{V}_{(123)}^{\rem}$ in \eqref{eqn: remainder wave of the higher order part for three-fold linearization}, by solving 
\begin{equation}\label{eqn: transport equation of the lower order symbol}
    \begin{aligned}
        \begin{cases}
            \mathscr{L}_{H_{\Box_g}}\sigma\l[\mathcal{V}_{(123)}^{\rem}\r] = -\imath h_1\sigma\l[\mathcal{V}_{(123)}^{\fre}\r],\quad \text{on } \Lambda_{(123)}^{g}\b\partial\Lambda_{(123)}^{g},\\
            \sigma\l[\mathcal{V}_{(123)}^{\rem}\r] = \mathscr{R}\l( \sigma[\Box_g]^{-1} \sigma\l[f_{(123)}^{\rem}\r]\r), \quad \text{on }\partial\Lambda_{(123)}^{g}.
        \end{cases}
    \end{aligned}
\end{equation}
Recall \eqref{eqn: the decomposition of f123fre and f123rem} and \eqref{eqn: the order of the principal symbol of f123fre}. We have 
$$
\sigma\l[f_{(123)}^{\rem}\r] \in S^{3\mu+2}(\Lambda_{(123)},\Omega_{1/2}\otimes L).
$$
Applying this to \eqref{eqn: transport equation of the lower order symbol} gives 
$$
 q := \sigma\l[\mathcal{V}_{(123)}^{\rem}\r] \in S^{3\mu+1/2}\l(\Lambda_{(123)}^{g},\Omega_{1/2}\otimes L\r).
$$
The local expression of $q$ on $\Lambda_{(123)}^{g}\cap U_j$ reads 
$$
q_j = \imath^{m_j}
(b_j \omega_j)|_{U_j \cap\Lambda_{(123)}^{g}},
$$
with a classical symbol $b_j$ on $U_j \cap\Lambda_{(123)}^{g}$, and $q_{j}$ satisfies
\begin{equation}\label{eqn: transition map for remainder symbol}
q_{j+1}(\beta_{\out}(s_j)) = \imath^{\sigma_{j+1,j}}q_{j}(\beta_{\out}(s_j)). 
\end{equation}
Then the transport equation \eqref{eqn: transport equation of the lower order symbol} along $\beta_{\out}([s_{j-1},s_{j}])$ reduces to 
$$
 \partial_{s}\l( e^{\rho_j(s)}b_j(\beta_{\out}(s))\r) = -\imath h_1(\gamma_{\out}(s)) e^{\rho_{j}(s)} a_{j}(\beta_{\out}(s)),\quad s\in (s_{j-1},s_{j}).
$$
 Since $e^{\rho_{j}(s)} a_{j}(\beta_{\out}(s))$ is constant for $s\in(s_{j-1},s_{j})$ due to \eqref{eqn: invariant}, we solve the equation with initial value $b_j(\beta_{\out}(s_{j-1}))$ as 
\begin{multline}\label{eqn: Lagrangian submanifold lower order symbol}
 \imath^{-m_j}e^{\rho_{j}(s_j)} \l(\omega^{-1}_jq_j\r)(\beta_{\out}(s_j)) \\
 = -\imath^{-m_j + 1} \l( \omega^{-1}_jp_j\r)(\beta_{\out}(s_{j-1})) \int_{s_{j-1}}^{s_j}h_1(\gamma_{\out}(s))ds 
 + \imath^{-m_j} \l(\omega^{-1}_j q_j\r)(\beta_{\out}(s_{j-1})).
\end{multline}
Combining \eqref{eqn: lagrangian submanifold 1}, \eqref{eqn: lagrangian submanifold 2} and \eqref{eqn: Lagrangian submanifold lower order symbol} leads to that 
$$
q_j (\beta_{\out}(s_j)) 
=  -\imath p_j (\beta_{\out}(s_j)) \int_{s_{j-1}}^{s_j}h_1(\gamma_{\out}(s))ds+ C_{j} q_j(\beta_{\out}(s_{j-1})) ,
$$
By \eqref{eqn: lagrangian submanifold 1} and \eqref{eqn: transition map for remainder symbol}, for $1\le j\le n-1$, we have 
\begin{align*}
    \lefteqn{q_{j+1}(\beta_{\out}(s_{j+1}))} \\ =& -\imath p_{j+1}(\beta_{\out}(s_{j+1}))\int_{s_j}^{s_{j+1}}h_1(\gamma_{\out}(s))ds + C_{j+1} q_{j+1}(\beta_{\out}(s_j))\\
    =&-\imath p_{j+1}(\beta_{\out}(s_{j+1}))\int_{s_j}^{s_{j+1}}h_1(\gamma_{\out}(s))ds + C_{j+1} \imath^{\sigma_{j+1,j}} q_{j}(\beta_{\out}(s_j))\\
    =&-\imath p_{j+1}(\beta_{\out}(s_{j+1}))\int_{s_j}^{s_{j+1}}h_1(\gamma_{\out}(s))ds \\
    &-\imath C_{j+1}\imath^{\sigma_{j+1,j}} p_{j}(\beta_{\out}(s_j))\int_{s_{j-1}}^{s_{j}}h_1(\gamma_{\out}(s))ds+ C_jC_{j+1}\imath^{\sigma_{j+1,j}}q_{j}(\beta_{\out}(s_{j-1}))\\
    =& -\imath p_{j+1}(\beta_{\out}(s_{j+1}))\int_{s_{j-1}}^{s_{j+1}} h_1(\gamma_{\out}(s))ds + C_{j} C_{j+1} \imath^{\sigma_{j+1,j}}q_{j}(\beta_{\out}(s_{j-1})).
\end{align*}
Iteratively, by \eqref{eqn: patching lagrangian submanifold 1}, we have  
\begin{multline*}
q_{n}(\beta_{\out}(s_n))=\sigma\l[\mathcal{V}_{(123)}^{\rem}\r](\c{z},\c{\zeta}) \\
=  -\imath \sigma\l[\mathcal{V}_{(123)}^{\fre}\r](\c{z},\c{\zeta}) \int_{0}^{s_{\out}}h_1(\gamma_{\out}(s))ds+ \alpha_{\out}(\c{z},\c{\zeta})\sigma\l[\mathcal{V}_{(123)}^{\rem}\r](\c{y},\c{\eta}).
\end{multline*}
We get \eqref{eqn: principal symbol higher-order remiander wave} by plugging in the expression of $\sigma\l[\mathcal{V}_{(123)}^{\fre}\r](\c{z},\c{\zeta})$ and $\sigma\l[\mathcal{V}_{(123)}^{\rem}\r](\c{y},\c{\eta})$ from \cite[(32),(42)]{Chen-Lu-Zhang-2025}.
\end{proof}

This lower order symbol calculus enables us to reconstruct the potential on $\supp h_3$.

\begin{proof}[Proof of  \eqref{eqn: Recovery of the potential via cubic nonliearities}]

Let $u_{(123)}^{(j)}$ for $j=1, 2$ be in \eqref{eqn: three-fold linearization with twice observation} and
$$
    u_{(123)}^{(j)} := \mathcal{V}_{(123)}^{(j)} + \mathcal{W}_{(123)}^{(j)},\quad j=1,2
$$ as in \eqref{decomposition of three-fold linearization}.

Note that $h_3$ has been recovered in \eqref{h_3 is known} and $u_{(j)}^{\fre}$ in \eqref{eqn: equation of the free wave} is independent of every $h_k$ for $k \in \mathbb{Z}_+$. For the free part of $\mathcal{V}_{(123)}^{\fre}$ defined by \eqref{eqn: free wave of the higher order part for three-fold linearization}, we thus have $\mathcal{V}_{(123)}^{(1), \fre} = \mathcal{V}_{(123)}^{(2), \fre} $ indeed. Consequently, $u_{(123)}^{(1)} = u_{(123)}^{(2)}$ guarantees that 
\begin{align}\label{eqn : identity of u_{(123)}^{rem}}
\mathcal{V}_{(123)}^{(1),\rem}  + \mathcal{W}_{(123)}^{(1)}  = \mathcal{V}_{(123)}^{(2),\rem} + \mathcal{W}_{(123)}^{(2)}. 
\end{align}

 Due to \eqref{order of the two waves genereated by three-fold linearization} and \eqref{eqn: order of the higher-order remiander wave}, $\mathcal{W}_{(123)}^{(j)}$ is one order smoother than $\mathcal{V}_{(123)}^{(j),\rem}$. Taking the leading terms in \eqref{eqn : identity of u_{(123)}^{rem}} gives
$$
  \sigma\l[\mathcal{V}_{(123)}^{(1),\rem} \r](\c{z},\c{\zeta}) =  \sigma\l[\mathcal{V}_{(123)}^{(2),\rem} \r](\c{z},\c{\zeta}).
$$
In view of \eqref{eqn: principal symbol higher-order remiander wave}, we have
$$
   h_{3}(\c{y})\int_{0}^{s_{\out}}h_1^{(1)}\l(\gamma_{\out}(s)\r)ds =    h_{3}(\c{y})\int_{0}^{s_{\out}}h_1^{(2)}\l(\gamma_{\out}(s)\r)ds.
$$
Dividing both sides by $h_3(\c{y}) \neq 0$, we deduce
$$
\int_{0}^{s_{\out}}h_1^{(1)}\l(\gamma_{\out}(s)\r)ds = \int_{0}^{s_{\out}}h_1^{(2)}\l(\gamma_{\out}(s)\r)ds,\quad \text{if } h_{3}(\c{y}) \neq 0.
$$
Differentiating both sides at $\c{y} = \gamma_{\out}(0)$, we obtain
$$
   h_1^{(1)}(\c{y}) = h_1^{(2)}(\c{y}),\quad \text{if } h_{3}(\c{y})\neq 0.
$$
Finally, we extend  this, by continuity, to $\supp(h_3)$, completing the proof.
\end{proof}

\subsection{Partial recovery of the potential via quadratic nonlinearity}\label{sec: Recovery of the potential via quadratic nonliearities}

Next, we prove \eqref{eqn: Recovery of the potential via quadratic nonliearities}, the partial uniqueness of $h_1$ on $\supp(h_2)\b\supp(h_3)$.% by performing the $m$-th order linearization on \eqref{eqn : semilinear wave with power series}.

To extract the useful information for inversion, we choose a microlocal cut-off $\chi$ near $(\c{y},\c{\eta})$ such that 
\begin{equation}\label{eqn: microlocal cut-off}
 \begin{aligned}
  \begin{cases}
 \WF\l(\chi \r) \cap \l(\bigcup_{j=1}^{3} N^{\ast}K_{(j)}\r)  = \emptyset,\\
 \WF(\chi) \cap \l(\bigcup_{1\le k<l\le 3} N^{\ast} \l( K_{(k)}\cap K_{(l)} \r)\r)  = \emptyset.
 \end{cases}
 \end{aligned}
\end{equation}
In contrast with the decomposition in \eqref{decomposition of three-fold linearization}, we re-split $u_{(123)}$ in \eqref{three-fold linearization} into 
\begin{equation}\label{eqn: splitting u_123 into singular and smooth terms}
u_{(123)} := v_{(123)} + w_{(123)},
\end{equation}
where $v_{(123)}$ and $w_{(123)}$  respectively solve 
\begin{align}
 \label{singular term} \l(\Box_g+h_1\r) v_{(123)} &=  \sum_{\{i,j,k\} =  \{1,2,3\}} -\chi\l( h_2 u_{(ij)}u_{(k)} \r),\\
 \label{smooth term} \l(\Box_g+h_1\r) w_{(123)} &=  \sum_{\{i,j,k\} = \{1,2,3\}} -(\id-\chi)\l( h_2 u_{(ij)}u_{(k)} \r) -6 h_3u_{(1)}u_{(2)}u_{(3)}.
\end{align}
Recall that the wavefront sets of the inhomogeneous terms in \eqref{singular term} and \eqref{smooth term} are contained in \eqref{eqn: wavefront set of three-fold linearization}. The wavefront set properties \eqref{eqn: first wavefront set property} and \eqref{eqn: second wavefront set property} and the fact that $h_3 \equiv 0 $ in a neighbourhood of $\c{y}$ show that the inhomogeneous terms in the equation \eqref{smooth term} have the property that
$$
\Gamma  \cap \WF\l((\id-\chi)\l(h_2 u_{(ij)}u_{(k)}\r)\r) = \emptyset,\quad \text{and}\quad \Gamma \cap \WF\l(h_3 u_{
(1)}u_{(2)}u_{(3)}\r) = \emptyset,
$$
where $\Gamma$ is the image of the bicharacteristic from $(\c{y},\c{\eta})$ to $(\c{z},\c{\zeta})$ defined in \eqref{eqn: image of the bicharacteristics}. By the propagation of singularities and the global hyperbolicity of $(M,g)$, we have
\begin{equation}\label{eqn: regularity of w_123}
\Gamma \cap \WF\l(w_{(123)}\r) = \emptyset.
\end{equation}

Since $u_{(123)}  = \mathcal{V}_{(123)} + \mathcal{W}_{(123)}$ by \eqref{decomposition of three-fold linearization} and $h_3 = 0$ in a neighborhood of $\c{y}$, Proposition \ref{prop: principal symbol of three wave interaction} gives that
$$
 u_{(123)}\in I^{3\mu-3/2}\l( \Lambda_{(123)},\Lambda_{(123)}^{g}\r)\quad \text{and}\quad \sigma\l[u_{(123)}\r](\beta_{\out}(s)) = \sigma\l[\mathcal{W}_{(123)}\r](\beta_{\out}(s)),
$$
where $s\in [0,s_{\out}]$ and $\beta_{\mathrm{out}}(s)$ is the bicharacteristic from $(\c{y},\c{\eta})$ to $(\c{z},\c{\zeta})$ in \eqref{eqn: bicharacteristics}.

Combining the decomposition $u_{(123)} = v_{(123)}+w_{(123)}$ in \eqref{eqn: splitting u_123 into singular and smooth terms} with the regularity of $w_{(123)}$ in \eqref{eqn: regularity of w_123}, it follows that
\begin{equation}\label{eqn: microlocal property of v_123}
  v_{(123)} \in I^{3\mu-\frac{3}{2}}\l( \Lambda_{(123)},\Lambda_{(123)}^{g}\r),\quad\sigma\l[v_{(123)}\r](\beta_{\mathrm{out}}(s)) = \sigma\l[\mathcal{W}_{(123)}\r](\beta_{\mathrm{out}}(s)),
\end{equation}

Next, we analyse the lower order symbol of $v_{(123)}$ to recover $h_1$.

We begin by analyzing the lower order symbols of $u_{(ij)}$ with $1\le i\neq j\le 3$. As before, the free part of $u_{(ij)}$ is defined to be the solution  $u_{(ij)}^{\fre}$ to the potential-free wave equation  
\begin{equation}\label{eqn: free wave of two-fold linearization}
  \Box_gu_{(ij)}^{\fre} = -2h_2 u_{(i)}^{\fre}u_{(j)}^{\fre}. 
\end{equation}

We  denote the remainder in $u_{(ij)}$ by
$$
   u_{(ij)}^{\rem} := u_{(ij)} -  u_{(ij)}^{\fre}. 
$$
Recall that $u_{(ij)}$ and $u_{(ij)}^{\fre}$ solve the equations \eqref{two-fold linearization} and \eqref{eqn: free wave of two-fold linearization} respectively. Moreover, $u_{(i)}$ and $u_{(j)}$ can be decomposed into free and remainder parts as in \eqref{eqn: equation of the free wave} and \eqref{eqn: definition of the remiander wave}. Applying $\Box_g+h_1$ to both sides above, we compute
\begin{multline}\label{eqn: the equation of u_ij^rem}
    \l(\Box_g+h_1\r)u_{(ij)}^{\rem}  = \l(\Box_g+h_1\r)u_{(ij)} -\Box_g u_{(ij)}^{\fre} - h_1u_{(ij)}^{\fre}\\
    = -2h_{2}\l(u_{(i)}^{\fre}u_{(j)}^{\rem} + u_{(i)}^{\rem}u_{(j)}^{\fre} + u_{(i)}^{\rem}u_{(j)}^{\rem}  \r) - h_1u_{(ij)}^{\fre}:= F_{(ij)}^{\rem}.
\end{multline}
%\subsection{lower order symbol calculus of two-fold linearization}

To analyse the lower order symbol of $v_{(123)}$, we first define the free wave $v_{(123)}^{\fre}$ as the solution to 
\begin{equation}\label{eqn: equation for v_123^fre}
\Box_g v_{(123)}^{\fre} = \sum_{\{i,j,k\} = \{1,2,3\}} -\chi\l( h_2 u_{(ij)}^{\fre}u_{(k)}^{\fre} \r).
\end{equation}
Define the remainder wave by
$$
  v_{(123)}^{\rem} := v_{(123)} - v_{(123)}^{\fre}.
$$
Applying \(\Box_g + h_1\) to both sides, and invoking  \eqref{singular term} and \eqref{eqn: equation for v_123^fre}, we arrive at
\begin{equation}\label{eqn: remainder term of the wave produced by h_2}
\l(\Box_g+h_1 \r)v_{(123)}^{\rem}  =   \l(\Box_g+h_1 \r)v_{(123)} - \Box_g v_{(123)}^{\fre} - h_1v_{(123)}^{\fre} = F_{(123)} - h_1v_{(123)}^{\fre},
\end{equation}
where
$$
F_{(123)} :=\sum_{\{i,j,k\} = \{1,2,3\}} -\chi h_2\l( u_{(ij)}^{\rem}u_{(k)}^{\fre}  + u_{(ij)}^{\fre}u_{(k)}^{\rem} + u_{(ij)}^{\rem}u_{(k)}^{\rem}\r).
$$

\begin{proposition}\label{prop: principal symbol of v_123^rem}
    For $v_{(123)}^{\rem}$ solving the equation \eqref{eqn: remainder term of the wave produced by h_2}, we have
    the principal symbol
\begin{align}\label{principal symbol of v_123^rem}
\lefteqn{\sigma\l[v_{(123)}^{\rem}\r](\c{z},\c{\zeta})}  \\
\notag &= -\imath\sigma\l[v_{(123)}\r](\c{z},\c{\zeta})\l( \int_{0}^{s_{\out}}h_1(\gamma_{\out}(s))ds +\sum_{j=1}^{3}r^{2}\kappa_{(j)}^{-1}\int_{\gamma_{(j)}}h_1 \r).
\end{align}

\end{proposition}
Before proving Proposition \ref{prop: principal symbol of v_123^rem} for the principal symbol of $v_{(123)}^{\rem}$, we first derive the principal symbols of $u_{(ij)}^{\fre}$ and $u_{(ij)}^{\rem}$.

%\\
 %&= -\imath 2(2\pi)^{-2}h_{2}^{2}(\c{y})\mathcal{G}_{1}(\c{\eta}) \sigma\l[Q\r]\l(\c{z},\c{\zeta},\c{y},\c{\eta} \r)\prod_{j=1}^{3}\sigma\l[u_{(j)}\r]\l(\c{y},r^{-2}\kappa_{(j)}\c{\xi}_{(j)} \r)\\
 % &\quad \times \l( \int_{0}^{s_{\out}}h_1(\gamma_{\out}(s))ds + O(r^{2})\r),\quad\text{as }r\to 0+.

 Recall that $(\c{y},\c{\eta}_{(k)})$ lies in the wavefront set of $u_{(k)}$, where $\c{\eta}_{(k)}$ is defined in \eqref{eqn: notation of the linear relation}. By \cite[Lemma 2.7, p.141]{O}, the sum of two linearly independent light-like covectors is not light-like. Hence, 
$$
\l(\c{y},\c{\eta}_{(i)} +\c{\eta}_{(j)}\r) \notin \Char(\Box_g),\quad\text{for all } 1\le i \neq j\le3.
$$ Choose $\chi_{(ij)}$ as a microlocal cut-off near $(\c{y},\c{\eta}_{(i)}+\c{\eta}_{(j)})$ such that
\begin{equation}\label{eqn: property of the microlocal cut-off near y,eta_1+eta_2}
\WF\l( \chi_{(ij)}\r) \cap \l(N^{\ast}K_{(i)}\cup N^{\ast}K_{(j)} \r) = \emptyset,\quad\text{and}\quad\WF(\chi_{(ij)}) \cap \Char(\Box_g) = \emptyset.
\end{equation}

\begin{lemma}
Let $\chi_{(ij)}$ be the microlocal cut-off introduced in \eqref{eqn: property of the microlocal cut-off near y,eta_1+eta_2}. We have
\begin{equation}\label{eqn: order of free wave of two-fold linearization}
\chi_{(ij)} u_{(ij)}^{\fre} \in I^{2\mu-1}\l(N^{\ast}\l(K_{(i)}\cap K_{(j)} \r) \r)
\end{equation}
with principal symbol 
\begin{align}\label{eqn: principal symbol of two-fold linearization}
  \lefteqn{\sigma\l[ \chi_{(ij)} u^{\fre}_{(ij)}\r]\l(\c{y},\c{\eta}_{(i)}+\c{\eta}_{(j)} \r)}  \\ \notag  =& -\pi^{-1}h_{2}(\c{y})\l(\sigma\l[\Box_g\r]\l(\c{y},\c{\eta}_{(i)}+\c{\eta}_{(j)} \r)\r)^{-1} 
  \times\sigma\l[u_{(i)}\r]\l( \c{y},\c{\eta}_{(i)}\r) \sigma\l[u_{(j)}\r](\c{y},\c{\eta}_{(j)}).
\end{align}
\end{lemma}

\begin{proof}
Recall that $u_{(ij)}^{\fre}$ solves \eqref{eqn: free wave of two-fold linearization}. Then, \eqref{eqn: order of free wave of two-fold linearization} and \eqref{eqn: principal symbol of two-fold linearization} are given by Proposition \ref{prop: Greenleaf-Uhlmann} and \cite[Lemma 3.4]{LUW}.
\end{proof}

The following lemma computes the principal symbol of the remainder $u_{(ij)}^{\rem}$, which gives the lower order symbols of $u_{(ij)}$.

\begin{lemma}
  Let $\chi_{(ij)}$ be the microlocal cut-off defined in \eqref{eqn: property of the microlocal cut-off near y,eta_1+eta_2}. Then,
  \begin{equation}\label{eqn: order of the remainder wave of two-fold linearization}
    \chi_{(ij)} \l(u_{(ij)}^{\rem}\r) \in I^{2\mu-2}\l(N^{\ast}\l(K_{(i)} \cap K_{(j)}\r)  \r),
  \end{equation}
  with principal symbol
  \begin{multline}\label{eqn: principal symbol of the remiander wave of two fold-linearization}
 \sigma\l[u_{(ij)}^{\rem} \r]\l( \c{y},\c{\eta}_{(i)}+\c{\eta}_{(j)}\r) = \pi^{-1} h_{2}(\c{y})\sigma\l[\Box_g \r]^{-1}\l(\c{y},\c{\eta}_{(i)}+\c{\eta}_{(j)}\r)\\
 \times \l( \sigma\l[u_{(i)}\r]\l(\c{y},\c{\eta}_{(i)} \r)\sigma\l[u_{(j)}^{\rem}\r]\l(\c{y},\c{\eta}_{(j)}\r) +\sigma\l[u_{(j)}\r]\l(\c{y},\c{\eta}_{(j)} \r)\sigma\l[u_{(i)}^{\rem}\r]\l(\c{y},\c{\eta}_{(i)}\r) \r).
\end{multline}
\end{lemma}

\begin{proof}
Consider $i=1$ and $j=2$ without loss of generality. Recall that $u_{(12)}$ solves the equation \eqref{eqn: the equation of u_ij^rem} with inhomogeneous term
$$
 F_{(12)}^{\rem} = -2h_2\l( u_{(1)}^{\fre}u_{(2)}^{\rem} + u_{(1)}^{\rem}u_{(2)}^{\fre} + u_{(1)}^{\rem}u_{(2)}^{\rem}\r) - h_1 u_{(12)}^{\fre}.
$$

Applying Proposition \ref{prop: Greenleaf-Uhlmann} to $u_{(k)}^{\fre}$ and $u_{(k)}^{\rem}$ for $k=1,2$, together with the order of $u_{(12)}^{\fre}$ in \eqref{eqn: order of free wave of two-fold linearization}, gives that
\begin{align*}
   \chi_{(12)} \l(h_2 u_{(1)}^{\fre}u_{(2)}^{\rem} \r),   \chi_{(12)} \l(h_2 u_{(1)}^{\rem}u_{(2)}^{\fre} \r) &\in I^{2\mu}\l(N^{\ast}\l(K_{(1)}\cap K_{(2)}\r) \r),\\
   \chi_{(12)}\l(h_1 u_{(12)}^{\fre} \r),   \chi_{(12)} \l(h_2 u_{(1)}^{\rem} u_{(2)}^{\rem}\r) &\in I^{2\mu-1}\l(N^{\ast}\l(K_{(1)}\cap K_{(2)}\r) \r).
\end{align*}
Thus, the inhomogeneous term satisfies
\begin{equation}\label{eqn: order of the inhomogeneous term of the remainder wave of two-fold linearization}
  \chi_{(12)}F_{(12)}^{\rem} = -2 \chi_{(12)}h_{2} \l( u_{(1)}^{\fre}u_{(2)}^{\rem} + u_{(1)}^{\rem}u_{(2)}^{\fre} \r) \mod I^{2\mu-1}\l( K_{(1)}\cap K_{(2)}\r).
\end{equation}
Recall the solution operator $Q$ in \eqref{eqn: causal inverse}, which is the solution operator of the linear wave equation \eqref{eqn: linear relation}. Now decompose $\chi_{(12)}u_{(12)}^{\rem}$ into  
$$ 
\chi_{(12)}u_{(12)}^{\rem} =  \chi_{(12)}Q\l( \chi_{(12)} F_{(12)}^{\rem}\r) + \chi_{(12)}Q\l( \l(\id-\chi_{(12)}\r) F_{(12)}^{\rem}\r).
$$

 We claim that the second term $\chi_{(12)}Q\l( \l(\id-\chi_{(12)}\r) F_{(12)}^{\rem}\r)$ is smooth.
 In fact, by the property of $\chi_{(12)}$ in \eqref{eqn: property of the microlocal cut-off near y,eta_1+eta_2}, for all $(y,\eta)\in \WF(\chi_{(12)})$,  we have
$$
(y,\eta) \notin \WF\l(\l(\id-\chi_{(12)}\r)F_{(12)}^{\rem}\r),\quad\text{and}\quad (y,\eta)  \notin \Char(\Box_g).
$$
 \cite[Proposition 5.1.1]{D} yields that 
$$
\WF\l(\chi_{(12)} \r) \cap \WF\l( Q\l(\l( \id-\chi_{(12)}\r) F_{(12)}^{\rem} \r)\r) = \emptyset.
$$
Then the claim holds. 

Moreover, $Q$ acts as an elliptic pseudo-differential operator of order $-2$ away from $\Char\l[  \Box_g\r]$. In view of the order of $\chi_{(12)}F_{(12)}^{\rem} $ as in \eqref{eqn: order of the inhomogeneous term of the remainder wave of two-fold linearization}, we obtain that 
$$
   \chi_{(12)}Q\l(\chi_{(12)}F_{(12)}^{\rem} \r) \in I^{2\mu-2}\l(N^{\ast}\l(K_{(1)}\cap K_{(2)}\r) \r).
$$
Then, we have finished the proof of \eqref{eqn: order of the remainder wave of two-fold linearization}.

It remains to calculate the principal symbol of $u_{(12)}^{\rem}$. By \eqref{eqn: the equation of u_ij^rem}, 
$$
  \sigma\l[ u_{(12)}^{\rem}\r]\l(\c{y},\c{\eta}_{(1)} + \c{\eta}_{(2)}\r) = \l(\sigma\l[\Box_{g} \r]^{-1} \sigma\l[F_{(12)}^{\rem}\r]\r)\l(\c{y},\c{\eta}_{(1)} + \c{\eta}_{(2)}\r).
$$
Applying Proposition \ref{prop: Greenleaf-Uhlmann} and the expression of $F_{(12)}^{\rem}$ in \eqref{eqn: order of the inhomogeneous term of the remainder wave of two-fold linearization} 
 to compute the symbol of $F_{(12)}^{\rem}$, we obtain the expression in \eqref{eqn: principal symbol of the remiander wave of two fold-linearization}, which completes the proof.
\end{proof}

%\subsection{lower-order symbol calculus of three-fold linearizations}

We also need the microlocal structure of the inhomogeneous term $F_{(123)}$.

\begin{lemma}\label{lemma: microlocal structure of F_123}
 The distribution $F_{(123)}$ belongs to $I^{3\mu-1}\l(\Lambda_{(123)}\r)$, and modulo lower-order terms, it has the expression
$$
  F_{(123)} = \sum_{\{i,j,k \} = \{1,2,3\}} -\chi h_2 \l(\chi_{(ij)}\l( u_{(ij)}^{\rem} \r)u_{(k)}^{\fre}+ \chi_{(ij)} \l( u_{(ij)}^{\fre}\r) u_{(k)}^{\rem}  \r)\mod I^{3\mu-2}\l(\Lambda_{(123)} \r),
$$
where $\chi_{(ij)}$ are the microlocal cut-offs introduced in \eqref{eqn: property of the microlocal cut-off near y,eta_1+eta_2}.
\end{lemma}

\begin{proof}
 Consider the summand with $(i,j,k) = (1,2,3)$ as a representative. We first decompose
$$
    \chi\l(h_2 u^{\rem}_{(12)}u_{(3)}^{\fre} \r)  = \chi\l(h_2 \chi_{(12)} \l(u^{\rem}_{(12)}\r) u_{(3)}^{\fre} \r)\\
     + \chi\l(h_2 \l(\id-\chi_{(12)} \r)\l(u^{\rem}_{(12)}\r) u_{(3)}^{\fre} \r). 
$$
The second term on the RHS is smooth, by \cite[pp. 2216]{CLOP}; see also the proof of \cite[Lemma 4.9]{KLOU}. From the order of $u_{(3)}^{\fre}$ in \eqref{eqn: microlocal property of free linear waves} and the order of $u_{(12)}^{\rem}$ in \eqref{eqn: order of the remainder wave of two-fold linearization}, Proposition \ref{prop: Greenleaf-Uhlmann} implies
\begin{align*}
  \chi\l(h_2 u_{(3)}^{\fre}\chi_{(12)} \l(u^{\rem}_{(12)}\r) \r) \in I^{3\mu-1}\l(\Lambda_{(123)} \r).
\end{align*}
Analogously, using the order of $u_{(12)}^{\fre}$ in \eqref{eqn: order of free wave of two-fold linearization} and the order of $u_{(3)}^{\rem}$ as in \eqref{eqn: lower order symbols}, we obtain
 \begin{align*}
  \chi\l(h_2 u^{\fre}_{(12)}u_{(3)}^{\rem} \r)  &= \chi\l(h_2 u_{(3)}^{\rem} \chi_{(12)}\l(u_{(12)}^{\fre}\r)\r) + C^{\infty} \in I^{3\mu-1}\l(\Lambda_{(123)} \r),\\
    \chi\l(h_2 u^{\rem}_{(12)}u_{(3)}^{\rem} \r)  &= \chi\l(h_2 u_{(3)}^{\rem}\chi_{(12)}\l(u_{(12)}^{\rem} \r)\r) + C^{\infty} \in I^{3\mu-2}\l(\Lambda_{(123)} \r).
\end{align*}
Then, we complete the proof.
\end{proof}

Now we turn back to Proposition \ref{prop: principal symbol of v_123^rem}.

\begin{proof}[Proof of Proposition \ref{prop: principal symbol of v_123^rem}]

According to Lemma \ref{lemma: microlocal structure of F_123} and Proposition \ref{prop: Greenleaf-Uhlmann}, the principal symbol of $F_{(123)}$ is given by
\begin{align} \label{initial version of principal symbol of F_123}\begin{aligned}
\lefteqn{\sigma\l[F_{(123)}\r](\c{y},\c{\eta})} \\
  &= \sum_{\{i,j,k\} = \{1,2,3\}} -(2\pi)^{-1}h_{2}(\c{y})\sigma\l[ u_{(ij)}^{\rem}\r]\l(\c{y},\c{\eta}_{(i)}+\c{\eta}_{(j)}\r)\sigma\l[u_{(k)}\r]\l(\c{y},\c{\eta}_{(k)}\r)\\
  &\quad +\sum_{\{i,j,k\} = \{1,2,3\}} -(2\pi)^{-1}h_{2}(\c{y})\sigma\l[ u_{(ij)}\r]\l(\c{y},\c{\eta}_{(i)}+\c{\eta}_{(j)}\r)\sigma\l[u^{\rem}_{(k)}\r]\l(\c{y},\c{\eta}_{(k)}\r)\\
  &= \sum_{\{i,j,k\} = \{1,2,3\}}2(2\pi)^{-1}h^{2}_{2}(\c{y})\sigma\l[ \Box_g\r]^{-1}\l(\c{y},\c{\eta}_{(i)}+\c{\eta}_{(j)}\r)\mathcal{B}\l(\c{y},\c{\eta}_{(1)},\c{\eta}_{(2)},\c{\eta}_{(3)}\r),\end{aligned}
\end{align}
where by \eqref{eqn: principal symbol of two-fold linearization} and \eqref{eqn: principal symbol of the remiander wave of two fold-linearization}, the term $\mathcal{B}\l(\c{y},\c{\eta}_{(1)},\c{\eta}_{(2)},\c{\eta}_{(3)}\r)$ is defined by
\begin{align*}
  \mathcal{B}\l(\c{y},\c{\eta}_{(1)},\c{\eta}_{(2)},\c{\eta}_{(3)}\r)  :&=  \sigma\l[u_{(i)}^{\rem}\r]\l(\c{y},\c{\eta}_{(i)}\r)\sigma\l[u_{(j)}\r]\l(\c{y},\c{\eta}_{(j)}\r)\sigma\l[u_{(k)}\r]\l(\c{y},\c{\eta}_{(k)}\r)\\
  &+\sigma\l[u_{(i)}\r]\l(\c{y},\c{\eta}_{(i)}\r)\sigma\l[u_{(j)}^{\rem}\r]\l(\c{y},\c{\eta}_{(j)}\r)\sigma\l[u_{(k)}\r]\l(\c{y},\c{\eta}_{(k)}\r)\\
  &+\sigma\l[u_{(i)}\r]\l(\c{y},\c{\eta}_{(i)}\r)\sigma\l[u_{(j)}\r]\l(\c{y},\c{\eta}_{(j)}\r)\sigma\l[u_{(k)}^{\rem}\r]\l(\c{y},\c{\eta}_{(k)}\r).
\end{align*}
Since $(i,j,k)$ ranges over all permutations of $(1,2,3)$, we may use the homogeneity relation \eqref{relation between the principal symbol of free wave and remiander wave by homogeneity} to write 
\begin{equation}\label{eqn: the expression for B}
  \mathcal{B}\l(\c{y},\c{\eta}_{(1)},\c{\eta}_{(2)},\c{\eta}_{(3)}\r)  = -\imath\prod_{j=1}^{3}\sigma\l[u_{(j)}\r]\l(\c{y},\c{\eta}_{(j)} \r) \sum_{j=1}^{3}r^{2}\kappa_{(j)}^{-1}\int_{\gamma_{(j)}}h_1.
\end{equation}
Recall the definition of the symbol factor $\mathcal{G}_{1}$ from \eqref{eqn: factor derived from the symbol of Q}. We substitute \eqref{eqn: the expression for B} into \eqref{initial version of principal symbol of F_123} to obtain 
\begin{equation}\label{expression for the principal symbol of F_123}
\sigma\l[F_{(123)} \r](\c{y},\c{\eta})  = 2(2\pi)^{-1} h_{2}^{2}(\c{y}) \mathcal{G}_{1}(\c{y},\c{\eta})
\mathcal{B}\l(\c{y},\c{\eta}_{(1)},\c{\eta}_{(2)},\c{\eta}_{(3)}\r).
\end{equation}

Since the solution $v_{(123)}$ to \eqref{singular term} lies in $I^{3\mu-3/2}\l(\Lambda_{(123)},\Lambda_{(123)}^{g} \r)$, we use Lemma \ref{lemma: microlocal structure of F_123} to obtain
\begin{equation}
  \begin{aligned}
    \begin{cases}
     	\mathscr{L}_{H_{\Box_g}}\sigma\l[v_{(123)}^{\rem}\r] = -\imath h_1 \sigma\l[v_{(123)}\r],\quad &\text{on }\Lambda_{(123)}^{g}\b\partial\Lambda_{(123)}^{g},\\
 			\sigma\l[v_{(123)}^{\rem}\r] = \mathscr{R}\l(\sigma[\Box_g]^{-1}\sigma\l[F_{(123)}\r]\r),\quad &\text {on }\partial\Lambda_{(123)}^{g}.
 			\nonumber
    \end{cases}
  \end{aligned}
\end{equation}

Next, we solve this transport equation along $\beta_{\out}$. By the proof of Lemma \ref{lemma: microlocal structure of the higher order remiander wave of three fold-linearization}, there exists a zeroth order strictly positive factor $\alpha_{\out}$ such that
\begin{align}
\label{eqn:first expression of the principal symbol of v_123^rem}\lefteqn{\sigma\l[v_{(123)}^{\rem}\r](\c{z},\c{\zeta})}\\
\notag& = -\imath\sigma\l[v_{(123)}\r](\c{z},\c{\zeta})\int_{0}^{s_{\out}}h_1(\gamma_{\out}(s))ds + \alpha_{\out}(\c{z},\c{\zeta})\sigma\l[v_{(123)}^{\rem}\r](\c{y},\c{\eta}).
\end{align}
Recall the expression for $\sigma\l[v_{(123)}\r]$ in \eqref{eqn: microlocal property of v_123} and \eqref{expression for the lower order principal symbol of three-fold linearization}. By the expression for $\sigma\l[F_{(123)}\r]$ in \eqref{expression for the principal symbol of F_123} and Remark \ref{eqn: principal symbol calculus in different languages}, we have 
\begin{align}
\notag\alpha_{\out}(\c{z},\c{\zeta})\sigma\l[v_{(123)}^{\rem}\r](\c{y},\c{\eta}) &= \alpha_{\out}(\c{z},\c{\zeta})\mathscr{R}\l( \sigma\l[ \Box_g\r]^{-1} \sigma\l[ F_{(123)}\r]\r)(\c{y},\c{\eta})\\
\notag&=\sigma[Q](\c{z},\c{\zeta},\c{y},\c{\eta})\sigma\l[F_{(123)}\r](\c{y},\c{\eta})\\
\label{eqn:second expression of the principal symbol of v_123^rem}&=-\imath\sigma\l[v_{(123)} \r](\c{z},\c{\zeta})\sum_{j=1}^{3}r^{2}\kappa_{(j)}^{-1}\int_{\gamma_{(j)}}h_1.
\end{align}
Therefore, \eqref{principal symbol of v_123^rem} holds by combining \eqref{eqn:first expression of the principal symbol of v_123^rem} and \eqref{eqn:second expression of the principal symbol of v_123^rem}.
\end{proof}

We invoke the lower order symbol calculus to reconstruct $h_1$ on $\supp(h_2) \setminus \supp(h_3)$.

\begin{proof}[Proof of \eqref{eqn: Recovery of the potential via quadratic nonliearities}]
We first assume $h_{2}(\c{y})\neq 0$. Recall the decomposition of $u_{(123)}$ into $v_{(123)}$ and $w_{(123)}$ given in \eqref{singular term} and \eqref{smooth term}, and the equation for $v_{(123)}^{\fre}$ in \eqref{eqn: equation for v_123^fre}. Since $u_{(123)}$ is uniquely determined by the source-to-solution map $L_{N}$, and $h_3$ is known by Proposition \ref{h_3 is known}, the difference
$$
   v_{(123)}^{\rem} + w_{(123)} = u_{(123)} - v^{\fre}_{(123)}
$$
is uniquely determined by $L_{h_1,H}$. Moreover, $(\c{z},\c{\zeta})$ does not lie in the wavefront set of $w_{(123)}$ as in \eqref{eqn: regularity of w_123}. Namely, $w_{(123)}$ is microlocally smooth near $(\c{z},\c{\zeta})$ and thus negligible in the sense of singularities. Consequently, $\sigma\l[v_{(123)}^{\rem}\r]$   is uniquely determined by $L_{h_1,H}$. It follows from \eqref{principal symbol of v_123^rem} that 
$$
   \l(h_{2}^{(1)}(\c{y})\r)^{2}\int_{0}^{s_{\out}}h_1^{(1)}(\gamma_{\out}(s))\,ds =    \l(h_{2}^{(2)}(\c{y})\r)^{2}\int_{0}^{s_{\out}}h_1^{(2)}(\gamma_{\out}(s))\,ds. 
$$
By the assumption $\c{y}\in\supp(h_2) \setminus \supp(h_3)$ and the uniqueness \eqref{eqn: uniquness of the square of h_2}	of $|h^{2}|$ on $\D\b\supp(h_3)$, we have
$$
  \int_{0}^{s_{\out}} h_1^{(1)}(\gamma_{\out}(s))\,ds = \int_{0}^{s_{\out}} h_1^{(2)}(\gamma_{\out}(s))\,ds.
$$
Differentiating the integral at $s = 0$ yields $h_1^{(1)}(\c{y}) = h_1^{(2)}(\c{y})$.

If $\c{y} \in\supp(h_2) \b \supp(h_3)$ and $h_{2}(\c{y}) = 0$, a continuity argument concludes the proof. 
\end{proof}

\section{The $m$-wave linearization}

%\subsection{Inverse potential III : via higher order nonlinearity}\label{sec: Recovery of the potential via higher order nonlinearities}

This section is to recover the potential $h_1$ on the support of higher-order nonlinearities, which completes the proof of uniqueness of $h_1$. 

\begin{proposition}\label{prop: uniqueness of $V$ via higher order nonlinearity}Let $(M, g)$, $\D$, $\mho$, $N^{(j)}$, and $L_{N^{(j)}}$ be as in Theorem \ref{thm : simplified theorem}. Suppose that $L_{N^{(1)}} = L_{N^{(2)}}$. Then there holds that
    $$
         h_{1}^{(1)}(\c{y}) = h_{1}^{(2)}(\c{y}).
    $$
\end{proposition}

The proof is based on the interaction of $3$ distorted plane waves and $m-3$ Gaussian beams. The strategy is to take $m$-fold linearization with these waves, and measure the singularities of the $m$-linearized wave.

Specifically, consider $\c{y}\in\D\b\rotom$ satisfying that there exists $m \ge 4$ such that
	\begin{equation}\label{eqn: support condition}
		\c{y} \in \supp(h^{(k)}_m) \quad\text{and}\quad \c{y}\notin \mathop{\bigcup}_{j=2}^{m-1} \supp(h^{(k)}_j),\quad k = 1,2.
	\end{equation}

Construct three distorted plane waves $u_{(j)}\in I^{\mu}
\l(N^{\ast}\{\c{x}_{(j)}\}, N^{\ast}K_{(j)}\r)$ for $j = 1,2,3$ as in Section \ref{sec: three wave interactions}. Additionally, construct the auxiliary source $f_{\lambda,\c{x}_{(1)},\c{\xi}_{(1)}} \in C_{c}^{\infty}(\rotom)$ as in \eqref{eqn: construction of source of the Gaussian beam}. 
Write for $i\ge 4$,
\begin{equation}\label{eqn: notation of smooth wave}
 u_{(i)} := u_{\lambda,\c{x}_{(1)},\c{\xi}_{(1)}}\in C^{\infty}(M), \quad \text{where}\quad \l(\Box_g+h_1 \r) u_{\lambda,\c{x}_{(1)},\c{\xi}_{(1)}} = f_{\lambda,\c{x}_{(1)},\c{\xi}_{(1)}} .
\end{equation}
Moreover, $u_{\lambda,\c{x}_{(1)},\c{\xi}_{(1)}}$ satisfies the Taylor expansion \eqref{eqn: expansion of Gaussian beams at y} at $\c{y}$. For convenience, write $f_{\lambda} := f_{\lambda,\c{x}_{(1)},\c{\xi}_{(1)}} $ and $u_{\lambda}:= u_{\lambda,\c{x}_{(1)},\c{\xi}_{(1)}} $.

Denote $\e := (\e_{(1)},\e_{(2)},\dots,\e_{(m)})$ for small $\e_{(j)} > 0$. We enter in \eqref{eqn : semilinear wave with power series} the source
\begin{equation}\label{eqn: source to recover higher order nonlinearities}
  f = \e_{(1)}f_{(1)}+\e_{(2)}f_{(2)}+\e_{(3)}f_{(3)} + \sum_{j=4}^{m}\e_{(j)}f_{\lambda}.
\end{equation}
 The solution $u(\e)$ of the semilinear model with source \eqref{eqn: source to recover higher order nonlinearities} is smoothly dependent on $\e$. Moreover, the first-order linearization with source \eqref{eqn: source to recover higher order nonlinearities} is similar to \eqref{one-fold linearization}. That is, for $u_{(j)} = \partial_{\e_{(j)}}u(\e)|_{\e =0}$ with  $1\le j\le m$, it holds that
$$
  \l(\Box_g +h_1 \r)u_{(j)} =f_{(j)},
$$
where $f_{(j)} = f_{\lambda}$ for $ 4 \le j\le m$.

Next, we iteratively perform $l$-fold linearization for $l=2,\dots ,m$. For the set 
\begin{equation}\label{eqn: multi-index set}
\alpha = \{\alpha_1,\alpha_2,\dots,\alpha_l\}\subseteq \Z_{+}\quad \text{with}\quad 1\le \alpha_1<\alpha_2<\cdots<\alpha_l\le m,
\end{equation}
we denote by
$$
   u_{(\alpha)} := \partial_{\e_{(\alpha_1)}}\partial_{\e_{(\alpha_2)}}\cdots\partial_{\e_{(\alpha_l)}}u(\e)|_{\e=0}.
$$

For any non-empty subset $\beta\subseteq\alpha\subseteq \{1,\dots,m\}$, we also set
\[
u_{(\beta)}:=\prod_{\gamma\in\beta}\partial_{\varepsilon(\gamma)}u(\varepsilon)\big|_{\varepsilon=0}.
\]
For an integer $k\ge2$, we denote by $\mathcal Q_k(\alpha)$ the set of all partitions of $\alpha$ into $k$ non-empty blocks, that is
\[
\mathcal Q_k(\alpha)
 := \left\{\,  \{B_1,\dots,B_k\} \;\middle|\;
   \begin{array}{l}
     \varnothing \ne B_j \subset \alpha \ \text{for } 1\le j\le k,\\
     B_i \cap B_j = \varnothing \ \text{for } i\ne j,\text{ and }
     \displaystyle\bigcup_{j=1}^k B_j = \alpha
   \end{array}
  \right\}.
\]

Thus, $u_{(\alpha)}$ solves the $l$-fold linearized wave equation
\begin{align}
 \label{eqn: l-fold linearization} (\Box_g + h_1)u_{(\alpha)} &=  -\sum_{k=2}^{l}  h_k\l. \l(\partial_{\e_{(\alpha_1)}}\cdots\partial_{\e_{(\alpha_l)}} u^{k}(\e)\r)\r|_{\e=0}\\
  &
  \notag= -\sum_{k=2}^{l} k!\,h_k
     \sum_{\{B_1,\dots,B_k\}\in\mathcal Q_k(\alpha)}
      u_{(B_1)}\cdots u_{(B_k)}.
\end{align} 

%where $\{\beta_{k,j}\}_{j=1}^{k}$ is a family of multi-index sets such that \footnote{What is the difference between $\cup$ and $\sqcup$?}
%$$
%  \emptyset \neq\beta_{k,j} \subseteq \alpha \quad \text{and}\quad \bigsqcup_{j=1}^{k}\beta_{k,j} =\alpha, \quad\text{where}\quad 1\le j \le k, 2\le k \le l.
%$$
%The summation in \eqref{eqn: l-fold linearization} is overall partitions\footnote{What does this mean?} $\{ \beta_{k,j}\}_{j=1}^{k}$ of $\alpha$.

In particular, $$u_{(12\cdots m)} = \partial_{\e_{(1)}}\partial_{\e_{(2)}}\cdots\partial_{\e_{(m)}}u(\e)|_{\e = 0}$$ solves the $m$-fold linearized wave equation 
\begin{equation}\label{m-fold linearization}  
(\Box_g + h_1)u_{(12\cdots m)} =  - m!h_{m}u_{(1)}u_{(2)}u_{(3)}u_{\lambda}^{m-3}+\sum_{k=2}^{m-1}
  f_{k},
\end{equation}
where for convenience,
\begin{equation}\label{definition of the m-linearized source}
  f_{k} := -k! h_k \sum_{\{B_1,\dots,B_k\}\in\mathcal Q_k(\{1,2,\dots,m\})}
      u_{(B_1)}\cdots u_{(B_k)}.
\end{equation}
Recall that $\Gamma$ is the image of the bicharacteristic connecting $(\c{y},\c{\eta})$ and $(\c{z},\c{\zeta}).$
To detect the singularities of $u_{(12\cdots m)}$, we first prove
\begin{proposition}\label{prop: popagation of regularity}
  For $f_{k}$ in \eqref{definition of the m-linearized source} and $\Gamma$ in \eqref{eqn: image of the bicharacteristics}, it holds that
  $$
    \WF(f_k) \cap \Gamma = \emptyset.
  $$
\end{proposition}
To prove Proposition \ref{prop: popagation of regularity}, it is necessary to prove the following lemmas.

\begin{lemma}\label{lemma: first singularity analysis}
   Given that $u_{(\alpha)}$  in  \eqref{eqn: l-fold linearization} and  $\alpha \cap \{1,2,3\} = \emptyset$,  it holds that  $u_{(\alpha)}$ is smooth.
\end{lemma}

\begin{proof}

 We prove this lemma  by induction on $|\alpha|$.  For  $|\alpha| = 1$, without loss of generality,  let  $\alpha = \{j\} \subseteq\{4,5,\dots,m\}$. From \eqref{eqn: notation of smooth wave}, $u_{(\alpha)} = u_{(j)}$ is smooth.

Assume that  $u_{(\beta)}$ is smooth  for all index sets $\beta\subseteq \{4,\dots,m\}$  with $|\beta| < |\alpha|$. Recall that $u_{(\alpha)}$ solves the equation \eqref{eqn: l-fold linearization}  with inhomogeneous terms of the form $ h_k u_{(B_1)}u_{(B_2)}\cdots u_{(B_k)}$ where $\emptyset \neq B_1,\dots,B_k \subsetneq \alpha$. By the induction hypothesis,  $u_{(B_1)},\dots,u_{(B_k)}$ is smooth for all $j$.  Therefore, the right-hand side of \eqref{eqn: l-fold linearization} is smooth, which implies that $u_{(\alpha)}$ is smooth as well.
\end{proof}

\begin{lemma}\label{lemma: second singularity analysis}

   Let $u_{(\alpha)}$ be as defined in \eqref{eqn: l-fold linearization}. If $\alpha \cap \{1,2,3\} = \{j\}$, then
   $$
     \WF(u_{(\alpha)}) \subseteq  N^{\ast}K_{(j)}.
   $$

\end{lemma}

\begin{proof}

Let $j=1$ without loss of generality. We proceed by induction on $|\alpha|$. For $|\alpha|= 1$, it follows directly that $\WF\l( u_{(1)}\r) \subseteq N^{\ast}K_{(1)}$.

   Suppose  that $\WF(u_{(\beta)}) \subseteq N^{\ast}K_{(1)}$ for all index sets $\beta $ contained in $\{1,2,\dots, m\}$ such that $|\beta|<|\alpha|$  and  $\beta\cap\{1,2,3\} = \{1\}$. Recall that $u_{(\alpha)}$ is the solution of  \eqref{eqn: l-fold linearization}, that is, 
   $$
      (\Box_g+h_1) u_{(\alpha)} = -\sum_{k=2}^{\ell} k!\,h_k
     \sum_{\{B_1,\dots,B_k\}\in\mathcal Q_k(\alpha)}
      u_{(B_1)}\cdots u_{(B_k)}.
   $$
    Suppose that $1\in B_1$ without loss of generality. By the induction hypothesis and Lemma \ref{lemma: first singularity analysis},
   $$
   \WF\l(u_{(B_1)}\r) \subseteq N^{\ast}K_{(1)},\quad\text{and}\quad u_{(B_2)},\dots,u_{(B_k)}\in C^{\infty}(M).
   $$
   Hence, we have the inhomogeneous term  $\WF\l( u_{(B_1)}\cdots u_{(B_k)} \r)\subseteq N^{\ast}K_{(1)}$. Recall that $K_{(1)}$ is the light-cone emanating from $\c{x}_{(1)}$ and 
   $$
   N^{\ast}K_{(1)} = \{(x,\xi)\in T^{\ast}M \b 0; x\in K_{(1)},\xi \text{ is light-like}\}.
   $$
   Thus, we have $N^{\ast}K_{(1)} \subseteq \Char(\Box_g)$. It follows from \cite[Proposition 2.3]{GU} that
   $$
    \WF\l(u_{(\alpha)} \r) \subseteq N^{\ast}K_{(1)}.
   $$
   \end{proof}

\begin{lemma}\label{lemma: third singularity analysis}
   Let $u_{(\alpha)}$ be as defined in \eqref{eqn: l-fold linearization}. Suppose that $\alpha \cap \{1,2,3\} = \{j,k \}$. Then, the wavefront set of $u_{(\alpha)}$ satisfies
   $$
     \WF(u_{(\alpha)}) \subseteq N^{\ast}K_{(j)}\cup N^{\ast}K_{(k)} \cup N^{\ast}\l(K_{(j)}\cap K_{(k)} \r).
   $$
\end{lemma}
\begin{proof}
 Let $j=1,k=2$ without loss of generality. We prove that 
 $$
  \WF\l( u_{(\alpha)}\r) \subseteq N^{\ast}K_{(1)}\cup N^{\ast}K_{(2)} \cup N^{\ast}\l( K_{(1)}\cap K_{(2)}\r),
 $$
 if $\alpha \cap \{1,2,3\} = \{1,2\}$.
 
 We proceed by induction on $|\alpha|$. First, if $|\alpha| =2$, then
 $$
 u_{(\alpha)} = u_{(12)} = -2Q\l( h_2 u_{(1)}u_{(2)}\r),
 $$
 where $Q$ is the solution operator of the linear wave equation with potential defined in \eqref{eqn: causal inverse}. Since $u_{(j)} \in I^{\mu}\l( N^{\ast}K_{(j)}\r)$ for $j=1,2$, by \cite[Lemma 4.2]{KLOU},
 $$
   \WF\l(Q\l( u_{(1)}u_{(2)}\r)\r) \subseteq N^{\ast}K_{(1)}\cup N^{\ast}K_{(2)} \cup N^{\ast}\l( K_{(1)}\cap K_{(2)}\r).
 $$

Next, suppose that for all index sets $\beta \subseteq \{1,2,\dots,m\}$ satisfying $|\beta|<|\alpha|$ and $\beta\cap\{1,2,3\} = \{1,2\}$, the following inclusion holds:
 $$
 \WF(u_{(\beta)}) \subseteq N^{\ast}K_{(1)}\cup N^{\ast}K_{(2)} \cup N^{\ast}\l( K_{(1)}\cap K_{(2)}\r).
 $$ 
 Recall that $u_{(\alpha)}$ solves equation \eqref{eqn: l-fold linearization}. Thus, $u_{(\alpha)}$ can be expressed as 
  \begin{equation}\label{eqn: u_alpha as a linear combination}
    u_{(\alpha)} =  -\sum_{k=2}^{l}\sum_{\{B_1,\dots,B_k\}\in\mathcal Q_k(\alpha)} k!\,Q\l(h_k     
      u_{(B_1)}\cdots u_{(B_k)}\r).
   \end{equation}
 The proof reduces to analyzing the two possible cases of the partition $\{B_1,\dots,B_{k}\}$.

    \textbf{Case 1.} $1,2 \in B_j$ for some $B_j\in \{B_1,\dots,B_k\}$. Assume that $j=1$ without loss of generality. By the induction hypothesis, 
   $$
   \WF\l(u_{(B_1)}\r) \subseteq N^{\ast}K_{(1)}\cup N^{\ast}K_{(2)} \cup N^{\ast}\l(K_{(1)}\cap K_{(2)} \r).
   $$
  Moreover, $u_{(B_2)},\dots,u_{(B_k)} \in C^{\infty}(M)$ by Lemma \ref{lemma: first singularity analysis}. Thus,
   $$
      \WF\l(   u_{(B_1)}u_{(B_2)}\cdots u_{(B_k)}\r) \subseteq N^{\ast}K_{(1)}\cup N^{\ast}K_{(2)} \cup   N^{\ast}\l( K_{(1)}\cap K_{(2)}\r).
   $$
   Note that $N^{\ast}K_{(1)}, N^{\ast}K_{(2)}\subseteq \Char(\Box_g)$ and 
   $$
  \l( N^{\ast}\l( K_{(1)}\cap K_{(2)}\r)\b \l(N^{\ast}K_{(1)} \cup N^{\ast}K_{(2)} \r)   \r) \cap \Char\l( \Box_g\r) = \varnothing,
   $$
  due to the fact that the sum of two linearly independent light-like covectors is not light-like \cite[Lemma 27, pp 141]{O}. By \cite[Proposition 2.3]{GU},
  $$
    \WF\l( Q\l(h_k     
      u_{(B_1)}\cdots u_{(B_k)}\r)\r) \subseteq N^{\ast}K_{(1)}\cup N^{\ast}K_{(2)} \cup   N^{\ast}\l( K_{(1)}\cap K_{(2)}\r).
  $$

  \textbf{Case 2.} There exists $j_1 \neq j_2$ such that $1\in B_{j_1}$ and $2\in B_{j_2}$. Suppose $1\in B_1$ and $2\in B_2$ without loss of generality. By Lemma \ref{lemma: first singularity analysis} and Lemma \ref{lemma: second singularity analysis},
  $$
   \WF\l(u_{(B_1)} \r)\subseteq N^{\ast}K_{(1)},  \WF\l(u_{(B_2)} \r)\subseteq N^{\ast}K_{(2)}\quad\text{and}\quad u_{(B_{j})} \in C^{\infty}(M), \quad3\le j\le k.
  $$
 Thus, we use the wavefront set calculus  \cite[Theorem 8.2.10]{H1} to derive
   $$
      \WF\l( u_{(B_1)}u_{(B_2)}\cdots u_{(B_k)}\r) \subseteq N^{\ast}K_{(1)}\cup N^{\ast}K_{(2)} \cup   N^{\ast}\l( K_{(1)}\cap K_{(2)}\r).
   $$
   As in the first case, we have
   $$
    \WF\l( Q\l(h_k     
      u_{(B_1)}\cdots u_{(B_k)}\r)\r) \subseteq N^{\ast}K_{(1)}\cup N^{\ast}K_{(2)} \cup   N^{\ast}\l( K_{(1)}\cap K_{(2)}\r).
  $$

 Finally, recall the expression of $u_{(\alpha)}$ in \eqref{eqn: u_alpha as a linear combination}. We conclude the proof by combining the two cases.

\end{proof}

\begin{lemma}\label{lemma: fourth singularity analysis}
 Let $u_{(\alpha)}$ be as defined in \eqref{eqn: l-fold linearization} and $\Gamma$ be as defined in \eqref{eqn: image of the bicharacteristics}. Given that $ \{1,2,3 \} \subseteq \alpha $ and the support condition \eqref{eqn: support condition} holds, 
   $$
      \Gamma \cap \WF\l(u_{(\alpha)}\r) = \emptyset.
   $$
\end{lemma}
\begin{proof}
We proceed by induction on $|\alpha|$. First, consider the base case $|\alpha| = 3$. In this case, $u_{(\alpha)} = u_{(123)}$ solves the wave equation \eqref{three-fold linearization}. The wavefront set of the inhomogeneous terms
 $h_3u_{(1)}u_{(2)}u_{(3)}$ and $h_{2}u_{(ij)}u_{(s)}$, where $\{i,j,s\} = \{1,2,3\}$, is described in \eqref{eqn: wavefront set of three-fold linearization}. In particular, we have
$$
\Gamma \cap \WF\l(u_{(1)}u_{(2)}u_{(3)} \r) = \{(\c{y},\c{\eta})\},\quad\text{and}\quad \Gamma \cap \WF\l(u_{(ij)}u_{(s)} \r) = \{(\c{y},\c{\eta})\}.
$$
Since $h_2,h_3 \equiv 0$ in a small neighborhood of $\c{y}$ by the support condition \eqref{eqn: support condition}, $\Gamma$ does not intersect with the wavefront sets of $h_3 u_{(1)}u_{(2)}u_{(3)}$ and $h_2u_{(ij)}u_{(s)}$. By the theorem of propagation of singularities \cite[Theorem 6.1.1]{DH}, it follows that
$\Gamma \cap \WF\l(u_{(123)}\r) = \emptyset$.

Now assume the lemma holds for all $\beta\subseteq\{1,2,\dots,m\}$ with $|\beta|<|\alpha|$. By \eqref{eqn: l-fold linearization}, the wave $u_{(\alpha)}$ is expressed as 
 $$
    u_{(\alpha)} =  -\sum_{k=2}^{l}\sum_{\{B_1,\dots,B_k\}\in\mathcal Q_k(\alpha)} k!\,Q\l(h_k     
      u_{(B_1)}\cdots u_{(B_k)}\r).
 $$
   It suffices to prove that
   \begin{equation}\label{eqn: property of wave front set with higher order nonlinearity}
       \Gamma \cap \WF\l(Q\l(h_k     
      u_{(B_1)}\cdots u_{(B_k)}\r) \r) = \emptyset.
   \end{equation}

 The proof reduces to three cases of the partition $\{B_1,\dots,B_{k}\}$.

 \textbf{Case 1.} Suppose that $\{1,2,3\}\subseteq B_{j}$ for some $B_j \in \{B_1,\dots,B_k\}$. Without loss of generality, assume that $\{1,2,3\}\subseteq B_1$. By induction hypothesis,
$$
   \Gamma \cap \WF\l(u_{(B_1)}\r) = \emptyset.
$$
Furthermore, $u_{(B_{j})} \in C^{\infty}(M)$ for $2\le j \le k$ by Lemma \ref{lemma: first singularity analysis}. Thus, \eqref{eqn: property of wave front set with higher order nonlinearity} holds for the first case by the theorem of propagation of singularities in \cite[Theorem 23.2.9]{H3}.

\textbf{Case 2.} Suppose that there exists three disjoint index sets $B_{j_1},B_{j_2},B_{j_3}$ such that $1\in B_{j_1}, 2\in B_{j_2}, 3\in B_{j_3}$. Without loss of generality, assume that $1\in B_{1}, 2\in B_2$ and $3\in B_3$.

Then, by Lemma \ref{lemma: first singularity analysis} and \ref{lemma: second singularity analysis},
$$
  \WF\l( u_{(B_j)}\r) \subseteq N^{\ast}K_{
  (j)},\quad j=1,2,3,\quad\text{and}\quad u_{(B_{l})}\in C^{\infty}(M),\quad 4\le l\le k.
$$
Applying the wavefront set calculus  \cite[Theorem 8.2.10]{H1} yields
\begin{multline}\label{eqn: three product of the wave front set}
 \WF\l(u_{(B_1)}u_{(B_2)}\cdots u_{(B_k)} \r)   \\
  \subseteq \bigcup_{j=1}^{3} N^{\ast}K_{(j)} \cup \bigcup_{1\le j<k\le 3}
N^{\ast}\l( K_{(j)}\cap K_{(k)}\r) \cup N^{\ast}\l(\bigcap_{l=1}^{3} K_{(l)} \r).
\end{multline}
Hence, by \eqref{eqn: first wavefront set property} and \eqref{eqn: second wavefront set property},
$$
  \WF\l(u_{(B_1)}u_{(B_2)}\cdots u_{(B_k)} \r) \cap \Gamma = \{(\c{y},\c{\eta})\}.
$$
Since $h_k \equiv 0$ in a small neighborhood of $\c{y}$, it follows  that $\Gamma$ does not intersect with the wavefront set of $h_k u_{(B_1)}u_{(B_2)}\cdots u_{(B_k)}$. Thus, \eqref{eqn: property of wave front set with higher order nonlinearity} also holds in this case.

\textbf{Case 3.} Suppose there exist $B_{j_1}, B_{j_2}$ such that $p,q\in B_{j_1}$ and $l\in B_{j_2}$, where $(p,q,l)$ forms a permutation of $(1,2,3)$.

Without loss of generality, assume that $1,2\in B_1$ and $3\in B_{2}$. By Lemma \ref{lemma: third singularity analysis}, we have
$$
\WF\l(u_{(B_1)}\r) \subseteq N^{\ast}K_{(1)}\cup N^{\ast}K_{(2)} \cup N^{\ast}\l(K_{(1)} \cap K_{(2)} \r).
$$
In addition, by Lemma \ref{lemma: first singularity analysis} and Lemma \ref{lemma: second singularity analysis}, it follows that
$$
  \WF\l(u_{(B_2)}\r) \subseteq N^{\ast}K_{(3)},\quad\text{and}\quad u_{(B_l)} \in C^{\infty}(M), \quad 3\le l\le k.
$$
Therefore, the inclusion \eqref{eqn: three product of the wave front set} also holds by the wavefront set calculus. Thus, the property \eqref{eqn: property of wave front set with higher order nonlinearity} is verified as well analogous to Case 2.
\end{proof}

\begin{proof}[Proof of Proposition \ref{prop: popagation of regularity}]
    By \eqref{definition of the m-linearized source}, $f_{k}$ is expressed as
    $$
        f_{k} =   \sum_{\{B_1,\dots,B_k\}\in\mathcal Q_k(\{1,2,\dots,m\})}-k!h_k
      u_{(B_1)}\cdots u_{(B_k)}, \quad 2\le k \le m-1,
    $$
    where $h_j \equiv 0$ in a small neighborhood of $\c{y}$. It reduces to prove that for all $ 2\le k\le m-1$ and $ \{B_1,\dots,B_k\}\in \mathcal{Q}_{k}(\{1,2,\dots,m\})$,
    \begin{equation}\label{eqn: propagation of regularity for lower order}
        \Gamma \cap \WF\l(  h_k u_{(B_1)} u_{(B_2)}\cdots u_{(B_k)} \r) = \emptyset,\quad
    \end{equation}

\textbf{Case 1.}    Suppose there exists $B_{j}$ for some $1\le j \le k$  such that $1,2,3\in B_j$. Without loss of generality, assume that $1,2,3\in B_{1}$. By Lemma \ref{lemma: fourth singularity analysis}, we have $\WF\l(u_{(B_1)}\r) \cap \Gamma = \emptyset$. Moreover, $u_{(B_2)},\dots,u_{(B_k)}\in C^{\infty}$ in this case. Thus, we have proved \eqref{eqn: propagation of regularity for lower order} for the first case.

\textbf{Case 2.} Suppose there exist $B_{j_1},B_{j_2}$ for $1\le j_1 < j_2\le k$ such that $p,q\in B_{j_1}$ and $l\in B_{j_2}$, where $\{p,q,l\} = \{1,2,3\}$. Without loss of generality, we assume that $1,2 \in B_{1}$ and $3\in B_2$.

By Lemma \ref{lemma: third singularity analysis}, we have
$$
\WF\l( u_{(B_1)}\r) \subseteq N^{\ast}K_{(1)}\cup N^{\ast}K_{(2)} \cup N^{\ast}\l(K_{(1)}\cap K_{(2)} \r).
$$
Moreover, by Lemma \ref{lemma: first singularity analysis} and Lemma \ref{lemma: second singularity analysis}, we have
$$
 \WF\l(u_{(B_2)}\r) \subseteq N^{\ast}K_{(3)},\quad u_{(B_3)},\dots,u_{(B_k)}\in C^{\infty}.
$$
Applying the wavefront set calculus gives that
$$
 \WF\l(u_{(B_1)}u_{(B_2)}\cdots u_{(B_k)} \r)   
  \subseteq \bigcup_{j=1}^{3} N^{\ast}K_{(j)} \cup \bigcup_{1\le j<k\le 3}
N^{\ast}\l( K_{(j)}\cap K_{(k)}\r) \cup N^{\ast}\l(\bigcap_{l=1}^{3} K_{(l)} \r).
$$
By \eqref{eqn: first wavefront set property} and \eqref{eqn: second wavefront set property}, we have
$$
 \WF\l(u_{(B_1)}u_{(B_2)}\cdots u_{(B_k)} \r) \cap \Gamma = \{(\c{y},\c{\eta})\}.
$$
Since $h_k\equiv 0$ in a neighborhood of $\c{y}$, we have proved \eqref{eqn: propagation of regularity for lower order} for the second case.
    
\textbf{Case 3.} Suppose there exist $B_{j_1} ,B_{j_2},B_{j_3}$ for $1\le j_1<j_2<j_3 \le k$ such that $1\in B_{j_1}, 2\in B_{j_2}$ and $3\in B_{j_3}$.  Without loss of generality, we assume that $1\in B_1,2\in B_2$ and $3\in B_3$. By Lemma \ref{lemma: first singularity analysis} and Lemma \ref{lemma: second singularity analysis}, we have
$$
\WF\l(u_{(B_j)}\r) \subseteq N^{\ast}K_{(j)},\quad j=1,2,3,\quad\text{and}\quad u_{(B_4)},\dots, u_{(B_k)}\in C^{\infty}.
$$
According to the wavefront set calculus, we have
$$
 \WF\l(u_{(B_1)}u_{(B_2)}\cdots u_{(B_k)} \r)   
  \subseteq \bigcup_{j=1}^{3} N^{\ast}K_{(j)} \cup \bigcup_{1\le j<k\le 3}
N^{\ast}\l( K_{(j)}\cap K_{(k)}\r) \cup N^{\ast}\l(\bigcap_{l=1}^{3} K_{(l)} \r).
$$
Then, we have proved \eqref{eqn: propagation of regularity for lower order} by \eqref{eqn: first wavefront set property}, \eqref{eqn: second wavefront set property} and the fact that $h_k\equiv 0$ in a neighborhood of $\c{y}$.

\end{proof}

\begin{proof}[Proof of Proposition \ref{prop: uniqueness of $V$ via higher order nonlinearity}]
We first decompose $\D \b \rotom$ as 
$$
 \D \b \rotom = \bigcup_{j=2}^{\infty} S_{j},
$$
where $S_j$ are pairwisely non-overlapping sets defined by
\begin{align*}
  S_3 &:= \supp(h_3) ,\\
  S_2 &:= \supp(h_2) \b S_3,\\
  S_j & := \supp(h_j) \b \l( S_2\cup\cdots \cup S_{j-1}\r),\quad j\ge4.
\end{align*}
Since $h_1$ has been determined on $S_2$ and $S_3$ in Proposition \ref{prop : 3 wave}, it remains to determine $h_1$ on $h_j$ for $j\ge 4$.

We claim that, we have
\begin{equation}\label{eqn: coincidence of support from source-to-solution map}
L_{N^{(1)}} = L_{N^{(2)}} \Longrightarrow S_{j}^{(1)} = S_{j}^{(2)},\quad j\ge2.
\end{equation}
We prove \eqref{eqn: coincidence of support from source-to-solution map} by induction. By Proposition \ref{prop : 3 wave}, 
$$
 S_{2}^{(1)} = S_{2}^{(2)},\quad\text{and}\quad S_{3}^{(1)} = S_{3}^{(2)}.
$$
We assume that $S_{j}^{(1)} = S_{j}^{(2)}$ for $j=2,\dots,m-1$ and we prove that $S_{m}^{(1)} = S_{m}^{(2)}$.

Suppose that $\c{y}\in S_{m}^{(1)}$ and $\c{y}\notin S_{m}^{(2)}$. For any $\c{y}\in S_{m}^{(1)}$, the $m-$fold linearization \eqref{m-fold linearization} implies
\begin{align*}
\l(\Box_g + h^{(1)}_1\r)u^{(1)}_{(12\cdots m)} &=  - m!h^{(1)}_{m}u^{(1)}_{(1)}u^{(1)}_{(2)}u^{(1)}_{(3)}\l(u_{\lambda}^{(1)}\r)^{m-3}+\sum_{k=2}^{m-1}
  f^{(1)}_{k},\\
  \l(\Box_g + h^{(2)}_1\r)u^{(2)}_{(12\cdots m)} &=  - m!h^{(2)}_{m}u^{(2)}_{(1)}u^{(2)}_{(2)}u^{(2)}_{(3)}\l(u_{\lambda}^{(2)}\r)^{m-3}+\sum_{k=2}^{m-1}
  f^{(2)}_{k}.
\end{align*}
By hypothesis induction, we have $h_j^{(1)}\equiv h_{j}^{(2)}\equiv 0$ in a neighborhood of $\c{y}$. Thus, $\WF\l(f_{k}^{(j)}\r) \cap \Gamma = \varnothing$, where $\Gamma$ is defined by \eqref{eqn: image of the bicharacteristics}, for $j=1,2$ and $2\le k\le m-1$ by Proposition \ref{prop: popagation of regularity}.

Since $\c{y}\in S_{m}^{(1)}$, it holds that for the microlocal cut-off $\chi$ in \eqref{eqn: microlocal cut-off} whose microlocal support lies in a conic neighborhood of $(\c{y},\c{\eta})$,
$$
\chi\l(h^{(1)}_{m}u^{(1)}_{(1)}u^{(1)}_{(2)}u^{(1)}_{(3)}\l(u_{\lambda}^{(1)}\r)^{m-3}\r) \in I^{3\mu+2}\l(\Lambda_{(123)}\r).
$$
The coincidence of the source-to-solution map gives that
$$
  u_{(12\cdots m)}^{(1)} = u_{(12\cdots m)}^{(2)}. 
$$
Thus,
$$
\chi\l(h^{(2)}_{m}u^{(2)}_{(1)}u^{(2)}_{(2)}u^{(2)}_{(3)}\l(u_{\lambda}^{(2)}\r)^{m-3}\r) \in I^{3\mu+2}\l(\Lambda_{(123)}\r).
$$
Thus, $\c{y} \in S_{m}^{(2)}$. We have proved that $S_{m}^{(1)}\subseteq S_{m}^{(2)}$.

Analogously, it holds that $S_{m}^{(2)}\subseteq S_{m}^{(1)}$. Then, we have proved \eqref{eqn: coincidence of support from source-to-solution map}.

We have proved the uniqueness of $h_1$ on $S_2$ and $S_3$ in Proposition \ref{prop : 3 wave}. It reduces to prove the uniqueness of $h_1$ on $S_m$ for $m\ge 4$. 
For any $\c{y} \in \D\b \rotom$, there exists a unique $S_m$ for $m\ge 2$ such that $\c{y}\in S_{m}$.

\textbf{Case 1.} $\c{y}\in S_{m}$ and $h_m(\c{y})\neq0$. Recall $u_{(12\cdots m)}$ denotes the solution of the $m$-fold linearized wave equation \eqref{m-fold linearization}. We decompose the solution as
$u_{(12\cdots m)} = v_{(12\cdots m)} + w_{(12\cdots m)}$, 
where the components respectively satisfy
\begin{align*}
   \l(\Box_g+h_1 \r)v_{(12\cdots m)} &= -m! h_m u_{(1)}u_{(2)}u_{(3)}u_{\lambda}^{m-3} \\
   \l(\Box_g+h_1 \r)w_{(12\cdots m)} & =  f_2+\cdots +f_{m-1}.
\end{align*}

Since $u_{\lambda}\in C^{\infty}(M)$, the microlocal structure of $v_{(12\cdots m)}$ and $\mathcal{V}_{(123)}$ in \eqref{higher-order three-fold linearization} remains the same near the bicharacteristic from $(\c{y},\c{\eta})$ to $(\c{z},\c{\zeta})$. More precisely, by Proposition \ref{prop: principal symbol of three wave interaction}, 
$$
  v_{(12\cdots m)} \in I^{3\mu +\frac{1}{2}}\l(\Lambda_{(123)},\Lambda_{(123)}^{g} \r) \quad\text{away from } \cup_{j=1}^{3}N^{\ast}K_{(j)}.
$$
Analogous to \eqref{expression for the principal symbol of three-fold linearization}, the principal symbol of $ v_{(12\cdots m)}$ at $(\c{z},\c{\zeta})$ is
\begin{multline}\label{eqn: principal symbol of the m-fold linearization}
  \sigma\l[v_{(12\cdots m)}\r](\c{z},\c{\zeta}) = -m!\sigma[Q](\c{z},\c{\zeta},\c{y},\c{\eta}) \sigma\l[ h_mu_{(1)}u_{(2)}u_{(3)}u_{\lambda}^{m-3}\r](\c{y},\c{\eta})\\
  = -m!(2\pi)^{-2} h_{m}(\c{y}) u_{\lambda}^{m-3}(\c{y})\sigma[Q](\c{z},\c{\zeta},\c{y},\c{\eta}) \prod_{j=1}^{3}\sigma\l[u_{(j)}\r](\c{y},r^{-2}\kappa_{(j)}\c{\xi}_{(j)}).
\end{multline}

Moreover, by Proposition \ref{prop: popagation of regularity} and the theorem of propagation of singularities \cite[Theorem 6.1.1]{DH}, we have
$$
  (\c{z},\c{\zeta})\notin \WF\l(w_{(12\cdots m)}\r).
$$
Thus, the principal symbol coincides:
\begin{align*}
 \sigma\l[ u_{(12\cdots m)} \r](\c{z},\c{\zeta})  = \sigma\l[v_{(12\cdots m)} \r](\c{z},\c{\zeta}).
\end{align*}
Since $u_{(12\cdots m)}$ is uniquely determined by the source-to-solution map, it follows that for two independent measurement $L_{N^{(1)}}$ and $L_{N^{(2)}}$,
$$
  \sigma\l[v_{(12\cdots m)}^{(1)} \r](\c{z},\c{\zeta}) =   \sigma\l[v_{(12\cdots m)}^{(2)} \r](\c{z},\c{\zeta}).
$$
Substituting into the expression for the principal symbols for $v_{(12\cdots m)}^{(j)}$, and noting that $u_{\lambda}$ depends on $h_1$, we obtain
\begin{equation}\label{eqn: equation of nonlinearity times Gaussian beam}
   h^{(1)}_{m}(\c{y}) \l(u_{\lambda}^{(1)}(\c{y})\r)^{m-3} =  h^{(2)}_{m}(\c{y}) \l(u_{\lambda}^{(2)}(\c{y})\r)^{m-3} .
\end{equation}
Using the Taylor expansion \eqref{eqn: expansion of Gaussian beams at y}, for $j=1,2$, we have
$$
u^{(j)}_{\lambda}(\c{y}) = 1+O(\lambda^{-1}) \quad\text{as } \lambda\to+\infty.
$$ 
Taking the limit as $\lambda \to +\infty$ , we deduce that
$$
   h^{(1)}_{m}(\c{y}) = h^{(2)}_{m}(\c{y}).
$$
Since $h_{m}(\c{y})\neq 0$ by assumption, equation \eqref{eqn: equation of nonlinearity times Gaussian beam} further implies 
$$
\left(u_{\lambda}^{(1)}(\c{y})\right)^{m-3} = \left(u^{(2)}_{\lambda}(\c{y})\right)^{m-3}.
$$ 
Applying again the Taylor expansion of $u_{\lambda}(\c{y})$ from \eqref{eqn: expansion of Gaussian beams at y}, we obtain 
\begin{multline*}
   \left(1+\lambda^{-1}\left(b_{1,0}(\c{y})-\frac{\imath}{2}\int_{0}^{s_{\into}}h_1^{(1)}(\gamma_{(1)}(s))ds\right) + O(\lambda^{-2})\right)^{m-3} \\
   =\left(1+\lambda^{-1}\left(b_{1,0}(\c{y})-\frac{\imath}{2}\int_{0}^{s_{\into}} h_1^{(2)}(\gamma_{(1)}(s))ds \right)+ O(\lambda^{-2})\right)^{m-3},
\end{multline*}
where $\gamma_{(1)}$ is the light-like geodesic connecting $\c{x}_{(1)}$ and $\c{y}$. 
Thus, we have
$$
\int_{0}^{s_{\into}}h_1^{(1)}(\gamma_{(1)}(s))ds + O(\lambda^{-1})  = \int_{0}^{s_{\into}} h_1^{(2)}(\gamma_{(1)}(s))ds + O(\lambda^{-1}).
$$
Letting $\lambda \to +\infty$, we conclude that the truncated integral of the potential along the light-like geodesic from $\c{x}_{(1)}$ to $\c{y}$ is uniquely determined by the source-to-solution map, that is,
$$
 \int_{0}^{s_{\into}}h_1^{(1)}(\gamma_{(1)}(s))ds = \int_{0}^{s_{\into}} h_1^{(2)}(\gamma_{(1)}(s))ds,\quad\text{where  } \gamma_{(1)}(s_{\into}) = \c{y}.
$$
Differentiating both sides at $s = s_{\into}$, where $\gamma_{(1)}(s_{\into}) = \c{y}$, yields
$$
  h_1^{(1)}(\c{y}) = h_1^{(2)}(\c{y}).
$$

\textbf{Case 2.} $\c{y}\in S_m$ and $h_m(\c{y})=0$. Select a sequence $\{y_n\}\subseteq \{h_m\neq 0\}\b\cup_{j=2}^{m-1}\supp(h_j)$ such that $y_n\to \c{y}$ as $n\to\infty$. By the result established in Case 1, we have $h_1^{(1)}(y_n) = h_1^{(2)}(y_n)$ for all $n$. Taking the limit as $n\to \infty$ and using the continuity of $h_1^{(1)}$ and $h_1^{(2)}$, we conclude that
$h_1^{(1)}(\c{y}) = h_1^{(2)}(\c{y})$.
\end{proof}

\begin{proposition}
    $$
         h^{(1)}_1\equiv h^{(2)}_1\quad\text{on }\mathbb{D}.
    $$
\end{proposition}
\begin{proof}
  This proposition follows from the combination of \eqref{eqn: Recovery of the potential via cubic nonliearities}, \eqref{eqn: Recovery of the potential via quadratic nonliearities} and Proposition \ref{prop: uniqueness of $V$ via higher order nonlinearity}.
\end{proof}

\section{The $4$-wave linearization}
This section recovers $h_2$ and $h_4$ on $\D$. The proof closely follows \cite{LUW}, using the linearization via $4$ distorted plane waves.

\begin{proposition}\label{prop: recovery of h_2 and h_4}
  Given the condition of Theorem \ref{thm: main theorem}, 
  $$
    h^{(1)}_{2} = h_{2}^{(2)},\quad\text{and}\quad  h^{(1)}_{4} = h_{4}^{(2)}\quad\text{on }\D.
  $$
\end{proposition}

By Lemma \ref{lemma: broken light-like geodesics}, there exists a triple $(\c{x},\c{y},\c{z})$ such that $\c{x},\c{z}\in \rotom$ and $\c{y}\in \D\b \rotom$, and that there exists a light-like geodesic connecting $\c{x},\c{y}$, and that there exists a light-like geodesic connecting $\c{y},\c{z}$. Moreover, these two light-like geodesics only intersect at $\c{y}$. 

Let $\c{x_{(1)}} :=\c{x}$. Then, we can choose $\c{x}_{(2)}$, $\c{x}_{(3)}$ and $\c{x}_{(4)}$ such that there exists a light like geodesic from $\c{x}_{(j)}$ to $\c{y}$ and that $\c{x}_{(j)}$ lies on the same Cauchy-hypersurface.

Denote $\e = (\e_{(1)},\e_{(2)},\e_{(3)},\e_{(4)}) \in \R^4$ with $|\e|$ small. Construct the source of the semi-linear wave equation \eqref{eqn : semilinear wave with power series} as
  \begin{equation}\label{eqn: source for four wave interactions}
    f = \e_{(1)}f_{(1)} + \e_{(2)}f_{(2)}+\e_{(3)}f_{(3)}+\e_{(4)}f_{(4)}.
  \end{equation}
  Each $f_{(j)}$ generates a distorted plane wave
  $$
     u_{(j)} \in I^{\mu}\l(N^{\ast}\{\c{x}_{(j)}\},N^{\ast}K_{(j)} \r), \quad\text{where}\quad \l(\Box_{g}+h_1\r)u_{(j)} = f_{(j)}.
  $$
  where $\c{x}_{(j)} \in \rotom$ lies on the same Cauchy surface $\{t\}\times N$. For all nonempty disjoint proper subsets $I,J \subseteq \{1,2,3,4\}$, $\cap_{i\in I} K_{(i)}$ and $\cap_{j\in J}K_{(j)}$ intersect transversally with codimension $4 - |I| - |J|$. In particular, 
  $$
   \mathop{\bigcap}_{j=1}^{4} K_{(j)} = \{\c{y}\}.
  $$

  Let $u(\e)$ be the solution of \eqref{eqn : semilinear wave with power series} with the source \eqref{eqn: source for four wave interactions}. By \cite[(3.8)]{LUW}, the four-fold linearization of $u(\e)$ reads
 \begin{align*}
      u_{(1234)} =& \l.\partial_{\e_{(1)}}\partial_{\e_{(2)}}\partial_{\e_{(3)}}\partial_{\e_{(4)}}u\r|_{\e = 0}\\
      =& -4\sum_{\sigma \in S(4)} Q\l(h_2 u_{(\sigma(1))}Q\l(h_2 u_{(\sigma(2))} Q\l( h_2 u_{(\sigma(3))} u_{(\sigma(4))}\r) \r) \r)\\
      &-\sum_{\sigma \in S(4)} Q\l(h_2Q\l(h_2 u_{(\sigma(1))}u_{(\sigma(2))}\r) Q\l(h_2 u_{\sigma(3)} u_{(\sigma(4))}\r) \r)\\
       &+2\sum_{\sigma\in S(4)} Q\l(h_2 u_{(\sigma(1))} Q\l(h_{3}u_{(\sigma(2))}u_{(\sigma(3))}u_{(\sigma(4))}\r) \r)\\
       &+3\sum_{\sigma\in S(4)}Q\l(h_3 u_{(\sigma(1))}u_{(\sigma(2))}Q\l(h_2 u_{(\sigma(3))}  u_{(\sigma(4))}\r) \r)\\
       &-24 Q\l(h_4 u_{(1)} u_{(2)} u _{(3)} u_{(4)}\r).
 \end{align*}
  For convenience,  we write
 \begin{equation}\label{eqn: decomposition of fourth order linearization}
  \mathcal{V}_{(1234)} := -24Q\l(h_4 u_{(1)}u_{(2)}u_{(3)}u_{(4)} \r),\quad\text{and}\quad\mathcal{W}_{(1234)} := u_{(1234)} - \mathcal{V}_{(1234)}.
 \end{equation}
 
Let $\Lambda_{\c{y}}$ be the shorthand of $T^\ast_{\c{y}} M$, and
 \begin{align*}
   \Theta^{(1)} &:= \bigcup_{j=1}^{4} N^{\ast}K_{(j)}, \\  \Theta^{(3)} &:=  \bigcup_{1\le i<j<k\le 4}N^{\ast}\l(K_{(i)}\cap K_{(j)}\cap K_{(k)}\r),\\
       \Xi &:= \Theta^{(1)} \cup \Theta^{(3),g}\cup \Lambda_{\c{y}}.
 \end{align*}
  Denote by $\Theta^{(3),g}$ and $\Lambda_{\c{y}}^{g}$ and the future flowout of $\Theta^{(3)}$ and $\Lambda_{\c{y}}$ respectively.  

To recover $h_2$ and $h_4$, we employ the following results from \cite[Proposition 3.11, 3.12]{LUW}.
\begin{proposition}\label{prop: first proposition of four-fold linearization}
Given $(\c{y},\c{\eta}) \in \Lambda_{\c{y}} $ with $\c{\eta}$ light-like and future-pointing, $\c{\eta}$ has a unique decomposition $\c{\eta} = \sum_{j=1}^{4}\c{\eta}_{(j)} $ such that $\c{\eta}_{(j)} \in N^{\ast}_{\c{y}}K_{(j)}\b 0$. For $(\c{z},\c{\zeta})$ on the bicharacteristic emanating from $(\c{y},\c{\eta})$, it holds that $u_{(1234)} \in I^{4\mu+3}\l( \Lambda_{\c{y}}^{g} \b\Xi \r)$ with principal symbol
       $$
        \sigma\l[ u_{(1234)}\r]\l( \c{z},\c{\zeta}\r) = -24(2\pi)^{-3}\sigma\l[Q\r]\l(\c{z},\c{\zeta},\c{y},\c{\eta} \r) h_{4}(\c{y})\prod_{j=1}^{4}\sigma\l[u_{(j)}\r]\l(\c{y},\c{\eta}_{(j)} \r).
       $$

\end{proposition}

\begin{proposition}\label{prop: second proposition of four-fold linearization} In the same condition as Proposition \ref{prop: first proposition of four-fold linearization},
\begin{itemize}
    \item[(1)]
       If $h_3 \neq 0$ in a neiborhood of $\c{y}$, $\mathcal{W}_{(1234)}\in I^{4\mu - 1/2}\l(\Lambda_{\c{y}}^{g}\b\Xi \r)$ with principal symbol
       $$
          \sigma\l[ \mathcal{W}_{(1234)}\r](\c{z},\c{\zeta}) = (2\pi)^{-3}h_2(\c{y})h_3(\c{y})\mathcal{G}_{2}(\c{\eta}) \sigma\l[Q\r]\l(\c{z},\c{\zeta},\c{y},\c{\eta} \r)\prod_{j=1}^{4}\sigma\l[u_{(j)}\r]\l(\c{y},\c{\eta}_{(j)} \r),
       $$
       where $\mathcal{G}_{2}(\c{\eta})$ is a nonvanishing factor not related to $h_1$ and $H$. %with the expression
     %  $$
    %    \mathcal{G}_{3}(\c{\zeta}) = \sum_{\sigma\in S(4)}3\l|\c{\eta}_{(\sigma(1))}+\c{\eta}_{\sigma(2)}\r|^{-2}_{g(\c{y})} + 2\l|\c{\eta}_{(\sigma(2))}+\c{\eta}_{\sigma(3)} + \c{\eta}_{\sigma(4)}\r|^{-2}_{g(\c{y})}. 
    %   $$
\item[(2)]
        If $h_3 = 0$ in a neighborhood of $\c{y}$, it holds that $\mathcal{W}_{(1234)}\in I^{4\mu - 5/2}\l(\Lambda_{\c{y}}^{g}\b\Xi \r)$ with principal symbol
         $$
          \sigma\l[ \mathcal{W}_{(1234)}\r](\c{z},\c{\zeta}) = (2\pi)^{-3}h_{2}^{3}(\c{y})\mathcal{G}_{3}(\c{\eta}) \sigma\l[Q\r]\l(\c{z},\c{\zeta},\c{y},\c{\eta} \r)\prod_{j=1}^{4}\sigma\l[u_{(j)}\r]\l(\c{y},\c{\eta}_{(j)} \r),
       $$
        where $\mathcal{G}_{3}(\c{\eta})$ is a nonvanishing factor not related to $h_1$ and $H$.      % with the expression
     %  $$
      %  \mathcal{G}_{3}(\c{\zeta}) = \sum_{\sigma\in S(4)}\l(\l|\c{\eta}_{(\sigma(1))}+\c{\eta}_{\sigma(2)}\r|^{-2}_{g(\c{y})} + 4\l|\c{\eta}_{(\sigma(2))}+\c{\eta}_{\sigma(3)} + \c{\eta}_{\sigma(4)}\r|^{-2}_{g(\c{y})} \r)\l|\c{\eta}_{(\sigma(1))} + \c{\eta}_{(\sigma(2))} \r|^{-2}_{g(\c{y})}.
     %   $$
        \end{itemize}
\end{proposition}

Now we are in the position to reconstruct $h_2$ and $h_4$ completely.

\begin{proof}[Proof of Proposition \ref{prop: recovery of h_2 and h_4}] 

For $j=1,2$, let 
$$
  u^{(j)}_{(1234)} = \partial_{\e_{(1)}}\partial_{\e_{(2)}} \partial_{\e_{(3)}}\partial_{\e_{(4)}} \l(\l(L_{N^{(j)}}\l(\sum_{j=1}^{4}\e_{(j)}f_{(j)} \r)\r)\r|_{\e_{(1)} = \e_{(2)}=\e_{(3)}=\e_{(4)}=0}.
$$
The coincidence of the source-to-solution map leads to that
$$
  \sigma\l[u^{(1)}_{(1234)}\r](\c{z},\c{\zeta}) =  \sigma\l[u^{(2)}_{(1234)}\r](\c{z},\c{\zeta}). 
$$
Then we readily have 
$$
   h_{4}^{(1)}(\c{y}) = h^{(2)}_{4}(\c{y}),\quad \text{for all } \c{y}\in\D.
$$

With the knowledge of $ u_{(1)}, u_{(2)}, u_{(3)}, u_{(4)}$ and $h_{4}$,   $\mathcal{V}_{(1234)}$ in \eqref{eqn: decomposition of fourth order linearization} is uniquely determined. Then, $\mathcal{W}_{(1234)}$ is uniquely recovered by the source-to-solution map; namely, 
\begin{equation}\label{coincidence of W_1234}
 \mathcal{W}_{(1234)}^{(1)} = \mathcal{W}_{(1234)}^{(2)}.
\end{equation}

By \eqref{coincidence of W_1234} and Proposition \ref{prop: second proposition of four-fold linearization}, it holds that
$$
\sigma\l[\mathcal{W}_{(1234)}^{(1)}\r](\c{z},\c{\zeta}) = \sigma\l[\mathcal{W}_{(1234)}^{(2)}\r](\c{z},\c{\zeta}).
$$
If $h_{3}(\c{y}) \neq 0$, the first part of Proposition \ref{prop: second proposition of four-fold linearization} gives that
$$
   h^{(1)}_{2}(\c{y})h_{3}^{(1)}(\c{y}) = h_{2}^{(2)}(\c{y}) h_{3}^{(2)}(\c{y}).
$$
Since  $h_{3}^{(1)}(\c{y}) = h_{3}^{(2)}(\c{y})\neq0$, it follows that
$$
h^{(1)}_{2}(\c{y}) = h^{(2)}_{2}(\c{y}).
$$
If $\c{y} \in \supp(h_3)$ and $h_{3}(\c{y})=0$, we can choose a sequence in $\{h_{3} \neq 0\}$ converging to $\c{y}$ and get $h^{(1)}_{2}(\c{y}) = h^{(2)}_{2}(\c{y})$ by continuity.
If $h_{3} = 0$ in a neighborhood of $\c{y}$, the second part of Proposition \ref{prop: second proposition of four-fold linearization} gives that
$$
   \l(h^{(1)}_{2}(\c{y})\r)^{3} = \l(h_{2}^{(2)}(\c{y})\r)^{3}
$$
Then, $h_{2}^{(1)} = h_{2}^{(2)}$ away from the support of $h_3$.

\end{proof}

\section{The $m$-wave linearization (revisited)}\label{sec : Step 5}

We complete the proof by recovering the higher order terms, which is achieved by the linearization via $3$ distorted plane waves and $m-3$ Gaussian beams.

\begin{proposition}\label{prop: recovery of higher order nonlinearities}
  Given the condition of Theorem \ref{thm: main theorem}, for all $m\ge 5$
  $$
      h_{m}^{(1)} \equiv h_{m}^{(2)},\quad\text{on }\D.
  $$
\end{proposition}

\begin{proof}
Choose the source of the semi-linear wave equation as in \eqref{eqn: source to recover higher order nonlinearities}. Recall that $u_{(12\cdots m)}$ solves the equation \eqref{m-fold linearization}. In fact, 
$$
  u_{(12\cdots m)} = -m! Q\l( h_m u_{(1)}u_{(2)}u_{(3)}u_{\lambda}^{m-3}\r) + R_{m}\l(h_1,h_2,\dots, h_{m-1} \r),
$$
where 
$$
 R_{m}(h_1,h_2,\dots,h_{m-1}) = \sum_{k=2}^{m-1}Q(f_{k}),
$$
in which $f_k$ is defined in \eqref{definition of the m-linearized source} and $h_1,h_2,h_3,h_4$ is known on $\D$. Moreover, the finite speed of propagation gives that $\supp(u_{(j)}),\supp(u_{\lambda})\subseteq J^{+}(\rotom)$ and their values on $\D$ are uniquely determined by the source-to-solution map.

We recover $h_m$ with $m\ge 5$ by induction analogous to \cite[Section 3.4]{Hintz-Uhlmann-Zhai-IMRN}. Assume that $h_2,\dots,h_{m-1}$ is recovered. Hence, $R_{m}$ is known in $\D$ and is supported in $J^{+}(\rotom)$ by the finite speed of propagation. 
By Proposition \ref{prop: principal symbol of three wave interaction}, $u_{(12\cdots m)}$ is a conormal distribution on $I^{3\mu+1/2}\l(\Lambda_{(123)},\Lambda_{(123)}^{g}\r)$ away from $\cup_{j=1}^{3}N^{\ast}K_{(j)}$ with principal symbol
\begin{align*}
  \sigma\l[u_{(12\cdots m)} \r](\c{z},\c{\zeta}) =& -m!(2\pi)^{-2} h_{m}(\c{y}) u_{\lambda}^{m-3}(\c{y})\sigma[Q](\c{z},\c{\zeta},\c{y},\c{\eta}) \\
  &\times\prod_{j=1}^{3}\sigma\l[u_{(j)}\r](\c{y},r^{-2}\kappa_{(j)}\c{\xi}_{(j)}). 
\end{align*}

 For two independent measurement $L_{N^{(1)}}$ and $L_{N^{(2)}}$, we have

$$
 \sigma\l[u^{(1)}_{(12\cdots m)} \r](\c{z},\c{\zeta}) = \sigma\l[u^{(2)}_{(12\cdots m)} \r](\c{z},\c{\zeta}).
$$
Then, we have
$$
  h_{m}^{(1)} = h_{m}^{(2)}.
$$
 since $u_{\lambda}(\c{y})$ is known and nonvanishing.
\end{proof}

The proof of Theorem \ref{thm: main theorem} consists of the proofs of Proposition \ref{prop : 3 wave}, Proposition \ref{prop: uniqueness of $V$ via higher order nonlinearity}, Proposition \ref{prop: recovery of h_2 and h_4} and Proposition \ref{prop: recovery of higher order nonlinearities}.

\bigskip

\noindent {\bf Acknowledgements.}  The authors were supported in part by Natural Science Foundation of Shanghai 23JC1400501. 
Part of the work was done during X.C.’s visit to Research Center for Mathematics of Data (MoD) at
Friedrich-Alexander-Universit\"{a}t Erlangen-N\"{u}rnberg under the support of  NSFC-DFG Sino-German Mobility Programme M-0548. R. Z. thanks Professor Lauri Oksanen for discussions on Lorentzian geometry and inverse problems for nonlinear wave equations with higher-order nonlinearities.

\bigskip	\noindent {\bf Data Availability Statement.} Data sharing not applicable to this article as no datasets were generated or analysed during the current study.

\bigskip	\noindent {\bf Conflict of Interest.} The authors have no conflicts of interest to declare that are relevant to the content of this article.

\bibliographystyle{abbrv}
\bibliography{reference}

\end{document}